\documentclass{amsart}
\usepackage[english]{babel}
\usepackage{amsfonts, amsmath, amsthm, amssymb,amscd,indentfirst}

\usepackage{esint}
\usepackage{graphicx}
\usepackage{epstopdf}
\usepackage{amsmath,amssymb,latexsym,indentfirst}
\usepackage{times}
\usepackage{palatino}

\newtheorem{theorem}{Theorem}[section]

\newtheorem{proposition}{Proposition}[section]
\newtheorem{lemma}{Lemma}[section]
\newtheorem{definition}{Definition}[section]

\newtheorem{corollary}{Corollary}[section]

\newtheorem{conjecture}{Conjecture}[section]

\newtheorem{remark}{Remark}[section]

\def\R{\mathbb{R}}
\def\Rn{{\mathbb{R}}^n_+}
\def\d{\partial}
\def\e{\epsilon}
\def\a{\alpha}
\def\b{\beta}

\def\s{\sigma}

\def\bp{\begin{proof}}
\def\ep{\end{proof}}

\def\Bc{{\mathcal B}}
\def\Hc{{\mathcal H}}
\def\R{{\cal R}}

\def\e{{\rm e}}

\def\b{\beta}

\def\eh{{{\bf{\rm{e}}}}}

\def\Sc{{\mathcal S}}

\def\a{{\alpha}}

\def\s{{\sigma}}

\def\Mc{{\mathcal M}}

\def\Ac{{\mathcal A}}

\def\eh{{{\bf{\rm{e}}}}}

\def\a{{\alpha}}

\def\s{{\sigma}}

\def\R{\mathbb{R}}
\def\Rn{{\mathbb{R}}^n_+}
\def\d{\partial}
\def\e{\epsilon}
\def\a{\alpha}
\def\b{\beta}

\def\s{\sigma}

\def\eh{{{\bf{\rm{e}}}}}

\def\a{{\alpha}}

\def\s{{\sigma}}

\begin{document}

\title[A positive mass theorem]{A positive mass theorem for asymptotically flat manifolds with a non-compact boundary}

\author{S\'{e}rgio Almaraz}
\address{Universidade Federal Fluminense (UFF) - Instituto de Matem\'{a}tica\\
              Rua M\'{a}rio Santos Braga S/N  24020-140 Niter\'{o}i, RJ, Brazil}
              \email{almaraz@vm.uff.br}
\author{Ezequiel Barbosa}
\address{Universidade Federal de Minas Gerais (UFMG), Departamento de Matem\'{a}tica, Caixa Postal 702, 30123-970, Belo Horizonte, MG, Brazil}
\email{ezequiel@mat.ufmg.br}
\author{Levi Lopes de Lima}
\address{Universidade Federal do Cear\'a (UFC),
Departamento de Matem\'{a}tica, Campus do Pici, Av. Humberto Monte, s/n, Bloco 914, 60455-760,
Fortaleza, CE, Brazil.}
\email{levi@mat.ufc.br}
\thanks{The three authors were partially supported by  CNPq/Brazil grants.}

\begin{abstract}
We prove a positive mass theorem for $n$-dimensional asymptotically flat manifolds with a non-compact boundary if either $3\leq n\leq 7$ or if $n\geq 3$ and the manifold is spin. This settles, for this class of manifolds, a question posed in a recent paper by the first author in connection with the long-term behavior of a certain Yamabe-type flow on scalar-flat compact  manifolds with boundary.
\end{abstract}

\maketitle

\section{Introduction and statements of the results}\label{int}

Let $(M^n,g)$ be an oriented  Riemannian manifold  with a non-compact boundary $\Sigma$ and dimension $n\geq 3$.  We denote by $R_g$ the scalar curvature of $(M, g)$. We also assume that $\Sigma$ is oriented by an outward pointing unit normal vector $\eta$, so that  its mean curvature is $H_g=\text{div}_g\eta$.

We say that $(M,g)$  is {\em asymptotically flat} with decay rate $\tau>0$ if there exists a compact subset $K\subset M$ and a diffeomorphism $\Psi: M\setminus K\to \mathbb R^n_+\setminus \overline{B}_1^+(0)$ such that the following asymptotic expansion holds as $r\to +\infty$:
\begin{equation}\label{asymflat}
|g_{ij}(x)-\delta_{ij}|+r|g_{ij,k}(x)|+r^2|g_{ij,kl}(x)|=O(r^{-\tau}).
\end{equation}
Here, $x=(x_1,\cdots,x_n)$ is the coordinate system induced by $\Psi$, $r=|x|$, $g_{ij}$ are the coefficients of $g$ with respect to $x$, the comma denotes partial differentiation, $\mathbb R^n_+=\{x\in\mathbb R^n; x_n\geq 0\}$ and
$\overline{B}_1^+(0)=\{x\in\mathbb R^n_+; |x|\leq 1\}$. The subset $M_{\infty}=M\backslash K$ is called the {\em{end}} of $M$. In the following we use the Einstein summation convention with the index ranges $i,j,\cdots=1,\cdots, n$ and $\alpha,\beta,\cdots=1,\cdots,n-1$. Observe that, along $\Sigma$, $\{\partial_\alpha\}_\alpha$ spans $T\Sigma$ while $\partial_n$ points inwards. 

The simplest example, and in fact the model case, of a manifold in this class is the closed half-space $\mathbb R^n_+$ endowed with the standard flat metric $\delta$. This work is devoted to the study of a certain geometric invariant which measures the deviation  at infinity of a general asymptotically flat manifold $(M,g)$ from the model space $(\mathbb R^n_+,\delta)$.

\begin{definition}\label{defasymflat}
Suppose that  $\tau>(n-2)/2$ and $R_g$ and $H_g$ are integrable in $M$ and $\Sigma$, respectively.
In terms of  asymptotically flat coordinates as above, the {\em mass} of $(M,g)$ is given by
\begin{equation}\label{massadm}
{\mathfrak m}_{(M,g)}=\lim_{r\to +\infty}\left\{\int_{\Sc_{r,+}^{n-1}}(g_{ij,j}- g_{jj,i})\mu^i d\Sc_{r,+}^{n-1}+\int_{\Sc_r^{n-2}} g_{\alpha n}\vartheta^\alpha d\Sc_r^{n-2}\right\},
\end{equation}
where $\Sc_{r,+}^{n-1}\subset M$ is a large coordinate hemisphere of radius $r$  with outward unit normal  $\mu$, and $\vartheta$ is the outward pointing unit co-normal to $\Sc_r^{n-2}=\partial \Sc_{r,+}^{n-1}$, viewed as the boundary of the bounded region $\Sigma_r\subset \Sigma$ (see Figure \ref{fig1}). 
\end{definition}


\begin{figure}[!htb]
\centering
\includegraphics[scale=0.70]{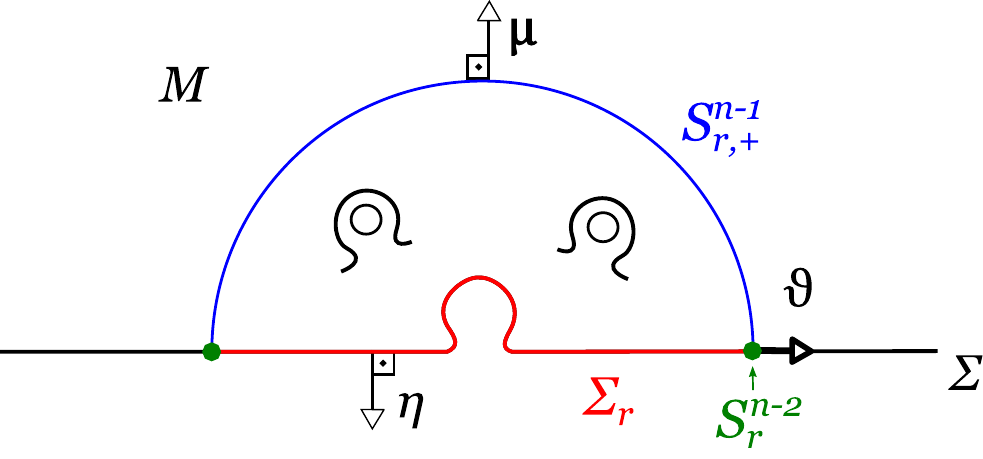}
\caption{An asymptotically flat manifold.}
\label{fig1}
\end{figure}

It  is not hard to verify that the limit on the right-hand side of (\ref{massadm}) exists. Moreover, we shall see in Section \ref{massgeo} that its value does not depend on the particular asymptotically flat coordinates chosen. Thus, $\mathfrak m_{(M,g)}$ is an invariant of the asymptotic geometry of $(M,g)$.

Besides having an obvious intrinsic geometric relevance, this invariant appears crucially in \cite{A} in connection with the global  convergence of a certain Yamabe-type flow first considered by S. Brendle in \cite{Br}, which produces scalar-flat compact domains with constant mean curvature boundary in the long-term limit. As explained in \cite{A}, the following conjecture is expected to be true.

\begin{conjecture}\label{conjpmt}
If $(M,g)$  is asymptotically flat with decay rate $\tau>(n-2)/2$ as above and satisfies $R_g\geq 0$ and  $H_g\geq 0$ then $\mathfrak m_{(M,g)}\geq 0$, with the equality occurring if and only if $(M,g)$ is isometric to $(\mathbb R^n_+, \delta)$.
\end{conjecture}

This conjecture has been confirmed in some special cases in \cite{Es, Ra}. In this work we show more generally that it holds true whenever the standard Positive Mass Conjecture holds (see \cite{SY1, SY2, Wi, Ba}). More precisely, the following result holds.

\begin{theorem}\label{mainm}
Conjecture \ref{conjpmt} holds true if either $3\leq n\leq  7$ or if $n\geq 3$ and $M$ is spin. 
\end{theorem}

Combined with the main result in \cite{A}, this guarantees the global convergence of the Yamabe-type  flow introduced in \cite{Br} for any initial scalar-flat compact manifold with boundary which meets the conditions of the theorem (i.e. either it is spin or has dimension $n\leq 7$).
The following 
immediate consequence of the rigidity statement in Theorem \ref{mainm} is also worth noticing. 

\begin{corollary}
\label{mainmcor}
Let $(M,g)$ be as in Theorem \ref{mainm} and assume further that there exists a compact subset $K\subset M$ such that $(M\setminus K,g)$ is isometric to $(\mathbb R^n_+\setminus \overline B_1^+(0),\delta)$. Then $(M,g)$ is isometric to $(\mathbb R^n_+, \delta)$.
\end{corollary}

We provide here two proofs of Theorem \ref{mainm}.  
In the proof presented in Section \ref{firstproof}, the first step  consists of an improvement of the asymptotics of the given metric in the spirit of the classical proof of the standard Positive Mass Theorem by Schoen and Yau (\cite{SY2}). Once this technical step is accomplished in Proposition \ref{conf:flat:assump}, this proof proceeds by a reduction to the classical cases via a doubling construction. We employ a result by Miao (\cite{Mi}), which covers the situation in which corners along a compact inner hypersurface appear, to prove that the conjecture above holds true whenever the classical Positive Mass Theorem holds for the doubled manifold. We also use the improvement in the asymptotics to present an alternative proof of Theorem \ref{mainm} in the case $3\leq n\leq 7$ which is more in the spirit of the classical arguments by Schoen-Yau (\cite{SY1}). More precisely, we show that the assumption of negative mass implies the existence of a stable minimal hypersurface without boundary leading to  a contradiction as in \cite{SY1}. In these proofs,
the rigidity statement in the theorem follows by means of the variational characterization  of the mass given in Proposition \ref{propvar}.
Finally, if $n\geq 3$ and $M$ is spin we present in Theorem \ref{wittenform} the natural extension of Witten's celebrated formula for the mass in terms of a suitable harmonic spinor globally defined on $M$. 
The proof of Theorem \ref{mainm} in this case is an immediate consequence of this expression.
\begin{remark}
\label{gener}
{\rm
We can conceive a version of Theorem \ref{mainm} in which the manifold $(M,g)$ has two collections of finitely many ends, say $\{E_l\}_{l=1}^m$ and $\{E'_{l}\}_{l=1}^{m'}$, which we assume endowed with diffeomorphisms $\Psi_l: E_l\to\mathbb R^n_+\setminus \overline{B}_1^+(0)$ and $\Psi'_{l}: E'_{l}\to\mathbb R^n\setminus \overline{B}_1(0)$ such that the expansion (\ref{asymflat}) holds. To each end $E_l$ we associate the mass given by (\ref{massadm}), and to each end $E'_l$ we associate its standard ADM  mass as in \cite{Ba,LP}. In this setting, the result says that if $R_g\geq 0$ and $H_g\geq 0$ then the mass of each end is non-negative. Moreover, if at least one mass vanishes then $(M,g)$ actually has only one end, being isometric either to $(\mathbb R^n_+,\delta)$ or to $(\mathbb R^n,\delta)$, according to the type of the end. The proofs of these more general statements follow by straightforward adaptations of the arguments presented here and therefore are omitted.
Observe that, since we are not assuming that $\Sigma$ is connected, we  allow for the presence of finitely many compact boundary components. If we think of $(M,g)$ as being the initial data set for a time-symmetric solution of Einstein fields equations, then these components may be viewed as (past or future) trapped hypersurfaces. In fact, the rigidity statement above actually implies that, in the presence of such compact trapped hypersurfaces, the mass of each end is actually positive. This is of course related to the positive mass theorem for black holes first considered in \cite{GHHP} (see also \cite{H}). For more recent results along these lines in the classical spin setting we refer to \cite{DX} and the references therein. }
\end{remark}

This paper is organized as follows. In Section \ref{mass} we give a motivation for the definition of the mass, by showing that it can be approached from a variational perspective. In Section \ref{massgeo} we prove that the mass is a geometric invariant in the sense that it does not depend on the asymptotic structure and varies smoothly with the metric. The first proof of Theorem \ref{mainm} is presented in Section \ref{firstproof}. As already mentioned, this proof makes use of a result due to Miao to reduce our positive mass theorem to the classical version, for manifolds without boundary. In Section \ref{anotherproof} we provide a second proof of our main theorem by adapting the arguments of Schoen and Yau, for dimensions up to seven, and the arguments of Witten, for spin manifolds.  The appendix is devoted to the proof of a technical result concerning some elliptic problems in weighted H\"{o}lder spaces.  

\vspace{0.2cm}

\noindent
{\bf Acknowledgements:} We first learned about the notion of mass in (\ref{massadm}) from Professor F. Marques. We would like to thank him for suggesting us that Theorem \ref{mainm} should hold true and for helpful discussions. The first and second authors would like to thank the hospitality of Professor A. Neves at Imperial College London where part of this research was carried out. While at Imperial College, the first author was supported by CAPES/Brazil and CNPq/Brazil grants and the second one by a CNPq grant.

\section{The variational approach to the mass}\label{mass}

In this section we show how the mass ${\mathfrak m}_{(M,g)}$ can be approached from a variational perspective. 
This not only motivates Definition \ref{defasymflat}  but also plays a key role in the proof of the rigidity statement in Theorem \ref{mainm}.  

The arguments here are similar to those used in  \cite[Section 8]{LP} for the ADM mass and we start by recalling this procedure.
We consider a manifold $M$ of dimension $n\geq 3$. Recall that $\Mc$, the space of Riemannian metrics on $M$, is an open cone in ${\rm Sym}^2(M)$, the space of bilinear symmetric tensors on $M$. Thus, if $g\in \Mc$  and $\delta g\in {\rm Sym}^2(M)$ is small enough then $g+\delta g\in \Mc$.
We recall that the corresponding variation for 
the  scalar curvature $R=R_g$ is
\[
\delta R=\nabla_i(\nabla_k\delta g^{ik}-\nabla^i{\overline \delta g})-R_{ik}\delta g^{ik},
\]
where $\overline \delta g =g^{ik}\delta g_{ik}$, $\nabla$ is the Levi-Civita connection of $g$ (extended to act on tensors) and $R_{ik}$ is the Ricci tensor. Also, the variation of the volume element is
\begin{equation}\label{liou}
\delta dM_g=\frac{1}{2}{\overline \delta g}dM_g.
\end{equation}
This allows us to compute the variation of the Hilbert-Einstein action given by
\[
g\in\Mc\mapsto \Ac(g)=\int_MRdM_g\,.
\]
We have
\begin{equation}\label{varac}
\delta \Ac=   \int_M \nabla_i(\nabla_k\delta g^{ik}-\nabla^i{\overline \delta g})dM_g -\int_M\left(
                   R_{ik}-\frac{R}{2}g_{ik}\right)\delta g^{ik} dM_g.
\end{equation}
Thus, if no boundary is present the first term in the right-hand side vanishes after integration by parts and we obtain the usual variational formula, namely,
\begin{equation}\label{eulerhe}
\delta \Ac=  -\int_M\left(
                   R_{ik}-\frac{R}{2}g_{ik}\right)\delta g^{ik} dM_g,
\end{equation}
In particular, 
critical metrics for the Hilbert-Einstein action are precisely Ricci-flat metrics.
This applies if $M$ is closed or, more generally, if the variation $\delta g$ is compactly supported.

If $M$ is asymptotically flat (with an empty boundary) then it is natural to consider variations preserving this kind of structure at infinity. This time a boundary contribution appears and, as explained in \cite{LP}, the ADM mass is precisely the term that should be subtracted from $\Ac$ to restore the expected form of the variational principle. More precisely, if for any such metric $g$ on $M$ we define the ADM mass of $(M,g)$ as 
\begin{equation}\label{admmass}
m_{(M,g)}=\lim_{r\to +\infty}\int_{\Sc_r}(g_{ij,j}-g_{jj,i})\mu^id\Sc_r^{n-1},
\end{equation}
where $\mu$ is the outward unit normal to a large coordinate sphere $\Sc_r^{n-1}$ in the asymptotic region, 
and set
\[
\Bc(g)=\Ac(g)-m_{(M,g)},
\]
then it follows from (\ref{varac}) that
\begin{equation}\label{euler}
\delta \Bc=-\int_M\left(R_{ik}-\frac{R}{2}g_{ik}\right)\delta g^{ik} dM_g,
\end{equation}
the obvious analogue of (\ref{eulerhe}).

Let us now assume that  $(M,g)$ is asymptotically flat with a non-compact boundary as in Theorem \ref{mainm}.
We claim that the natural analogue of $\Ac$ is the {\em Gibbons-Hawking-York action}  (\cite{GH, Y}) given by
\begin{equation}\label{ghyaction}
\widetilde\Ac(g)=\int_M RdM_g +2\int_\Sigma Hd\Sigma_h,
\end{equation}
where $h=g|_\Sigma$ and $H=H_g$.
As before, one must subtract the mass $\mathfrak m_{(M,g)}$ from this in order to restore the expected form of the variational principle. 

\begin{proposition}
\label{propvar}
If $(M,g)$ is asymptotically flat and 
\[
\widetilde \Bc=\widetilde  \Ac-{\mathfrak m}_{(M,g)},
\]
then
\begin{equation}\label{variform2}
\delta\widetilde \Bc= -\int_{M}\left(R_{ik}-\frac{R}{2}g_{ik}\right)\delta g^{ik}dM_g-\int_{\Sigma}\left(A_{\alpha\beta}-Hh_{\alpha\beta}\right)\delta h^{\alpha\beta}d\Sigma_h,
\end{equation}
where $A$ is the shape operator of $\Sigma$.
\end{proposition}

\begin{proof}
We adapt a classical computation (\cite{Ar, Lo}) to the quantity
\begin{equation}\label{ghyactionk}
\widetilde\Ac_r(g)=\int_{M_r} RdM_g +2\int_{\Sigma_r} Hd\Sigma_h,
\end{equation}
where $M_r$ is the compact domain whose boundary is $\Sigma_r\cup \Sc_{r,+}^{n-1}$; see Figure \ref{fig1}. Notice that this is {\em not} the standard GHY action for the compact manifold  $M_r$   since the boundary integral over $\Sc_{r,+}^{n-1}$
is missing.

In order to compute $\delta\widetilde \Ac_r$ we note that from (\ref{varac}) the variation of $\Ac_r$, the Hilbert-Einstein action evaluated on $M_r$, is
\begin{eqnarray}\label{eq:deltaA}
\delta \Ac_r & = &    \int_{\Sigma_r} \eta^i(\nabla_k\delta g^{k}_i-\nabla_i{\overline \delta g})d\Sigma_h
    +\int_{\Sc_{r,+}^{n-1}} \mu^i(\nabla_k\delta g^{k}_i-\nabla_i{\overline \delta g})dS_{r,+}^{n-1}\notag \\
     & & \quad -\int_{M_r}\left(
                   R_{ik}-\frac{R}{2}g_{ik}\right)\delta g^{ik} dM_g,
\end{eqnarray}
where $dS_{r,+}^{n-1}$ is the area element of $S_{r,+}^{n-1}$.
As usual,  
we adopt the index ranges $i,j,\cdots=1,\cdots,n$ and $\alpha,\beta,\cdots=1,\cdots,n-1$ and choose coordinates 
around $p\in \Sigma$
so that $\{\partial_\alpha\}_\alpha$ spans $T\Sigma$ while $\partial_n$ points inwards. Moreover, we assume that $\eta=-\partial_n$ at $p$.

Since
\begin{equation}
\label{unitnormalexp}
\eta=-(g^{nn})^{-1/2}g^{ni}\partial_i\,,
\end{equation}
the second fundamental form is
\begin{equation}\label{qual}
A_{\alpha\beta}=-\langle \eta,\nabla_\alpha\partial_\beta\rangle=(g^{nn})^{-1/2}\Gamma_{\alpha\beta}^n.
\end{equation}
and the mean curvature is 
\begin{equation}
\label{meancurvloc}
H =h^{\alpha\beta}A_{\alpha\beta}=\frac{1}{2}(g^{nn})^{1/2}
          (2g_{n\alpha,\alpha}-
             g_{\alpha\alpha,n}).
\end{equation}
This allows us to compute the variation of $H$. After evaluation at $p$, the final result is 
\[
\delta H=\nabla^\Sigma_\alpha\delta g^{\alpha}_n-\frac{1}{2}h^{\alpha\beta}\nabla_n\delta h_{\alpha\beta}-\frac{1}{2}H\delta g_{nn},
\] 
where $\nabla^\Sigma$ is the induced connection on $\Sigma$.  From this and (\ref{liou}) we see that 
\begin{equation}\label{varform}
\delta(2Hd\Sigma_h)=\left(2\nabla^\Sigma_\alpha\delta g^\alpha_n-h^{\alpha\beta}\nabla_n\delta h_{\alpha\beta}+H\delta h^{\alpha}_\alpha-H\delta g^n_ n\right)d\Sigma_h.
\end{equation}

On the other hand, at $p$ we have 
\begin{eqnarray*}
\eta^i(\nabla_k\delta g^{k}_i-\nabla_i{\overline \delta g}) & = & -\left(\nabla_\alpha\delta g^\alpha_n-\nabla_n{\overline \delta g}\right)\\
& = & -\nabla^\Sigma_\alpha\delta g_n^\alpha+\Gamma^\beta_{n \alpha}\delta g_\beta^\alpha-\Gamma^\alpha_{n\alpha}\delta g^n_n+
       \nabla_n(g^{\alpha\beta}\delta g_{\alpha\beta}),
\end{eqnarray*}
so that
\begin{equation}\label{inter}
\eta^i(\nabla_k\delta g^{k}_i-\nabla_i{\overline \delta g})=
  -\nabla^\Sigma_\alpha\delta g^\alpha_n +g^{\alpha\beta}\nabla_n\delta g_{\alpha\beta} -A^{\alpha\beta}\delta g_{\alpha\beta}+H\delta g_n^n.
\end{equation}
Thus, if we combine (\ref{inter}),  (\ref{varform}) and (\ref{eq:deltaA}) we get 
\begin{eqnarray*}
\delta\widetilde{\Ac}_r & = & -\int_{M_r}\left(R_{ik}-\frac{R}{2}g_{ik}\right)\delta g^{ik}dM_g-\int_{\Sigma_r}\left(A_{\alpha\beta}-Hh_{\alpha\beta}\right)\delta h^{\alpha\beta}d\Sigma_h \\
 & & \quad +\int_{\Sc_{r,+}^{n-1}} \mu^i(\nabla_k\delta g^{k}_i-\nabla_i{\overline \delta g})dS_{r,+}^{n-1}  -\int_{\Sigma_r} \nabla^{\Sigma}_\alpha \delta  g^\alpha_i\eta^i d\Sigma_h.
\end{eqnarray*}
The last integral is clearly a divergence so we can rewrite this as
\begin{eqnarray*}
\delta\widetilde{\Ac}_r & = & -\int_{M_r}\left(R_{ik}-\frac{R}{2}g_{ik}\right)\delta g^{ik}dM_g-\int_{\Sigma_r}\left(A_{\alpha\beta}-Hh_{\alpha\beta}\right)\delta h^{\alpha\beta}d\Sigma_h \\
 & & \quad +\int_{\Sc_{r,+}^{n-1}} \mu^i(\nabla_k\delta g^{k}_i-\nabla_i{\overline \delta g})dS_{r,+}^{n-1}  -\int_{\Sc_r^{n-2}} \vartheta^\alpha \delta  g_{\alpha i}\eta^i d\Sc_r^{n-2}.
\end{eqnarray*}
It follows from the results in the next section that  the last two integrals converge as $r\to +\infty$ to $\delta {\mathfrak m}_{(M,g)}$, the variation of the mass.
From this, (\ref{variform2}) follows easily.
\end{proof}

We  thus see that Ricci-flat metrics are again critical for $\widetilde \Bc$ with respect to variations fixing the metric along the boundary ($\delta h=0$). 

\section{The mass as a geometric invariant}\label{massgeo}

In this section we give a proof of the geometric invariance of the mass by adapting the standard arguments in \cite{Ba, LP}. We also show that this invariant depends smoothly on the asymptotically flat metric, thus justifying the computation leading to (\ref{variform2}). 

Let us define the function $r(x)$ as any smooth, positive extension of the asymptotic parameter $|x|$ to $M$.
We start by recalling the expansions of the scalar curvature and the mean curvature in the asymptotic region.

\begin{proposition}
\label{expasymreg}
One has 
\begin{equation}
\label{asymexpscal}
R=\mathcal{C}_{i,i} + \Theta, \quad \Theta=O(r^{-2\tau-2}),
\end{equation}
and 
\begin{equation}
\label{asympexpmean}
H  =  \frac{1}{2}\left(-\langle \mathcal{C},\eta\rangle+g_{n\alpha,\alpha}\right)+ \Theta',\quad \Theta'=O(r^{-2\tau-1}),
\end{equation}
where $\mathcal{C}_i=g_{ij,j}-g_{jj,i}$ is the ADM mass density.
\end{proposition}

\begin{proof}
The expansion (\ref{asymexpscal}) is well-known (see \cite{Ba, LP}). Also, (\ref{asympexpmean}) follows easily from (\ref{unitnormalexp}) and (\ref{meancurvloc}). 
\end{proof}

We now introduce the right functional spaces in order to handle this type of question. Given a complete Riemannian manifold $M$ (with or without boundary), $k\geq 0$ an integer and $\gamma\in\mathbb R$, 
we proceed as in \cite{LP} 
and define the {\em weighted} $C^k$ {\em space} $C^{k}_\gamma(M)$ as the set of $C^k$ functions $u$ on $M$ for which the norm  
\[
\|u\|_{C^k_\gamma(M)}=\sum_{i=0}^k\sup_M r^{-\gamma+i}|\nabla^iu|
\]
is finite. Moreover, if $0<\alpha<1$ we define the {\em weighted H\"older space} $C^{k,\alpha}_{\gamma}$ as the set of functions $u\in C^{k}_{\gamma}(M)$ such that   
the norm 
\[
\|u\|_{C^{k,\alpha}_\gamma(M)}=\|u\|_{C^k_\gamma(M)}+
\sup_{x,y}\left(\min r(x),r(y)\right)^{-\gamma+k+\alpha}\frac{|\nabla^ku(x)-\nabla^ku(y)|}{|x-y|^\alpha}
\]
is finite. Here, the supremum is over all $x\neq y$ such that $y$ is contained in a normal coordinate neighborhood of $x$, and $\nabla^ku(y)$ is the tensor at $x$ obtained by the pararel transport along the radial geodesic from $x$ to $y$.

We also define the {\em{weighted Lebesgue space}} $L^{q}_{0,\b}(M)$, $q\geq 1$, $\b\in \R$, as the set of locally integrable functions $u$  for which the norm
$$
\|u\|_{q,0,\b}=\left(\int_M|r^{-\b}u|^qr^{-n}dM_g\right)^{\frac{1}{q}}
$$
is finite. For $k\geq 0$ an integer, $q\geq 1$ and $\b\in \R$, we define the {\em{weighted  Sobolev  space}} $L^{q}_{k,\b}(M)$ to be the set of $u$ for which $|\nabla^iu|\in L^{q}_{0,\b-i}(M)$ for $i=0, 1, ..., k$, with the norm
$$
\|u\|_{q,k,\b}=\sum_{i=0}^{k}\|\nabla^i u\|_{q,0,\b-i}.
$$
Notice that for $\beta=-n/q$ we recover the standard Sobolev spaces, denoted simply by $L^q_k(M)$.

It is easy to check that these are Banach spaces whose underlying topologies do not depend on the involved choices. 
As stated in \cite{LP} for manifolds without boundary, the following weighted Sobolev lemma also holds in our context.
\begin{proposition}\label{Sobolev:lemma}
Let $q>1$,  $l-k-\a>n/q$ and $\e>0$. Then there are continuous embeddings $C^{l,\a}_{\b-\e}(M)\subset L^{q}_{l,\b}(M)\subset C^{k,\a}_{\b}(M)$. 
\end{proposition}
After fixing asymptotically flat coordinates on the end $M_{\infty}$ we consider, for each $\tau>0$, the space $\Mc_\tau$ of all metrics on $M$ so that 
\[
g-\delta\in C^{1,\alpha}_{-\tau}(M_{\infty}), 
\quad R\in L^1(M), \quad H\in L^1(\Sigma).
\] 
If we fix a background metric $g_0$ and write $g=g_0+b$, it is clear from Proposition \ref{expasymreg} that we can identify $\Mc_\tau$ to a subset of the affine space  
\[
\{g_0+b; b_{ij,ij}-b_{ii,jj}\in L^1(M_{\infty}), b_{\alpha\alpha,n}\in L^1(\Sigma\cap M_{\infty})\}.
\]
In the topology induced by this identification we have $g_k\to g$ if and only if 
\[
\|g_k-g\|_{C^{1,\alpha}_{-\tau}(M)}
\to 0
\]
and 
\[
\|R_{g_k}-R_g\|_{L^1(M)}+\|H_{g_k}-H_g\|_{L^1(\Sigma)}\to 0.
\]

The following proposition describes the main technical result on weighted H\"older spaces needed in this work.
Its proof is postponed to Appendix \ref{prooftech1}.
\begin{proposition}\label{isomorphism}
Let $(M,g)$ be an asymptotically flat manifold $(M,g)$ with $g\in \Mc_\tau$, $\tau>0$, and with a nonempty boundary $\Sigma$. Fix  $2-n<\gamma<0$ and let $T:C^{2,\a}_{\gamma}(M)\to C^{0,\a}_{\gamma-2}(M)\times C^{1,\a}_{\gamma-1}(\Sigma)$
be given by
$$
T(u)=\left(-\Delta_gu+hu, \d u/\d\eta+\bar{h}u\right), 
$$ 
where 
$\Delta_g$ is the Laplacian, $\eta$ is the outward unit normal to $\Sigma$,
$h\in C^{0,\a}_{-2-\e}(M)$ and $\bar{h}\in C^{1,\a}_{-1-\e}(\Sigma)$, for some $\e>0$ small.  
If $h\geq0$ and $\bar{h}\geq 0$ then $T$ is an isomorphism.
\end{proposition}

We can use standard interpolation methods to define $L^q_{k,\beta}(\Sigma)$ for any $k\in\mathbb R_+$. In particular, the restriction map
\[
u\in C_c^{\infty}(M)\mapsto (u, \partial u/\partial \eta )\in C_c^\infty(\Sigma)\times C_c^\infty(\Sigma)
\] 
extends continuously to the so-called {\em trace map}
\[
\mathfrak T:L^q_{2,\beta}(M)\to L^q_{2-1/q,\beta}(\Sigma)\times L^q_{1-1/q,\beta}(\Sigma),
\]
which is known to be surjective. 
Hence, it makes sense to consider the subspaces of $L^q_{2,\beta}(M)$ consisting of functions satisfying Neumann and Dirichlet boundary conditions, namely, 
\[
W_N=\{u\in L^{q}_{2,\b}(M);\:\d u/\d \eta=0\,\:\text{on}\:\Sigma\}
\]
and 
\[
W_D=\{u\in L^{q}_{2,\b}(M);\:u=0\,\:\text{on}\:\Sigma\}.
\]
\begin{proposition}\label{Sobolev:N}
Consider 
$\Delta_g:W_N\to  L^{q}_{0,\b-2}(M)$. Then\\
(a) $\Delta_g$ is an isomorphism if and only if $2-n<\b<0$;\\
(b) $\Delta_g$ is injective if $0>\b\notin \mathbb{Z}$;\\
(c) $\Delta_g$ is surjective if $2-n<\b\notin \mathbb{Z}$.
\end{proposition}
\begin{proposition}\label{Sobolev:D}
Consider $\Delta_g:W_D\to  L^{q}_{0,\b-2}(M)$. Then\\
(a) $\Delta_g$ is an isomorphism if and only if $2-n<\b<0$;\\
(b) $\Delta_g$ is injective if $0>\b\notin \mathbb{Z}$;\\
(c) $\Delta_g$ is surjective if $2-n<\b\notin \mathbb{Z}$.
\end{proposition}
\bp[Proofs of Propositions \ref{Sobolev:N} and \ref{Sobolev:D}]
We consider the the double $(\widetilde{M},\widetilde{g})$ of $(M,g)$ along $\Sigma$ defined by  $\widetilde{M}=M\times \{0,1\}\slash\thicksim$, where $(y,0)\thicksim (y,1)$ for all $y\in\Sigma$, and $\widetilde{g}(y,j)=g(y)$ for all $y\in M$ and $j=0,1$.  Although $\widetilde{g}$ is not smooth on $\widetilde{M}$, it satisfies the hypotheses  in \cite[Definition 2.1]{Ba}. Then both proofs follow from  \cite[Proposition 2.2]{Ba} by means of reflection arguments of functions on $M$. The details are left to the reader.
\ep
\begin{remark}
\label{remsurj}
{
\rm  Suppose $\b\notin \mathbb{Z}$, $\b>2-n$. As a consequence of Proposition \ref{Sobolev:N}, standard arguments show that the Neumann problem
\begin{equation}\label{sistSobolev}
\begin{cases}
\Delta_gu=f\,,&\text{in}\:M\,,
\\
\displaystyle\frac{\d u}{\d \eta}= \bar{f}\,,&\text{on}\:\Sigma\,,
\end{cases}
\end{equation}
has a solution $u\in L^{q}_{2,\b}(M)$ for any $f\in  L^{q}_{0,\b-2}(M)$ and $\bar{f}\in L^{q}_{1-1/q,\b}(\Sigma)$. 
In fact, we can solve the cases $f\equiv 0$ and $\bar{f}\equiv 0$ separately. The latter case follows directly from Proposition \ref{Sobolev:N}. 
In order to solve the case $f\equiv 0$, we choose $\phi\in L^{q}_{2,\b}(M)$ such that $\d\phi/\d\eta=\bar{f}$. Then we use Proposition \ref{Sobolev:N} to find $\psi\in W_N$ satisfying $\Delta_g \psi=-\Delta_g\phi\in L^{q}_{0,\b-2}$. Thus, $u=\psi+\phi$ is a solution to (\ref{sistSobolev}) when $f\equiv 0$.
A similar result holds for the Dirichlet problem in (\ref{sistSobolev}) as a consequence of Proposition \ref{Sobolev:D}.
}
\end{remark}

The geometric invariance of the mass is described in the next proposition. 

\begin{proposition}
\label{massinv}
If $(M,g)$ is asymptotically flat with $g\in\Mc_\tau$, $\tau>(n-2)/2$, then the  mass $\mathfrak m_{(M,g)}$ only depends on the metric $g$. Moreover, this dependence is smooth with respect to the topology on $\Mc_{\tau}$ described above. 
\end{proposition}

The rest of this section is devoted to the proof of Proposition \ref{massinv}.
The first step  is to integrate  (\ref{asymexpscal}) over the region $M_{r,r'}$ determined by two coordinates hemispheres, say $\Sc_{r,+}^{n-1}$ and $\Sc_{r',+}^{n-1}$, with $r<r'$. We find that
\begin{eqnarray*}
\int_{M_{r,r'}}R dM_g & = & \int_{\Sc_{r',+}^{n-1}}\langle \mathcal{C},\mu\rangle d\Sc_{r',+}^{n-1} -
    \int_{\Sc_{r,+}^{n-1}}\langle \mathcal{C},\mu\rangle d\Sc_{r,+}^{n-1}\\
    & & \quad +\int_{\Sigma_{r,r'}}\langle \mathcal{C},\eta\rangle d\Sigma_h + \int_{M_{r,r'}}\Theta dM_g, 
 \end{eqnarray*}
where $\Sigma_{r,r'}$ is the portion of the boundary of $M_{r,r'}$ lying on $\Sigma$. On the other hand, from (\ref{asympexpmean}) we get 
\begin{eqnarray*}
\int_{\Sigma_{r,r'}}\langle \mathcal{C},\eta\rangle d\Sigma_h & = & \int_{\Sc_{r'}^{n-2}}g_{\alpha n}\vartheta^{\alpha}d\Sc_{r'}^{n-2}-
\int_{\Sc_{r}^{n-2}}g_{\alpha n}\vartheta^{\alpha}d\Sc_{r}^{n-2}\\
& & \quad -2\int_{\Sigma_{r,r'}}Hd
\Sigma_h + 2\int_{\Sigma_{r,r'}}\Theta'd\Sigma_h.
\end{eqnarray*}
Hence,
\begin{align}\notag
\int_{\Sc_{r',+}^{n-1}}\langle \mathcal{C},\mu\rangle d\Sc_{r',+}^{n-1}
&+\int_{\Sc_{r'}^{n-2}}g_{\alpha n}\vartheta^{\alpha}d\Sc_{r'}^{n-2} 
\\
=&\int_{M_{r,r'}}R dM_g 
+ 2\int_{\Sigma_{r,r'}}H d\Sigma_h\notag
\\
&\quad+ \int_{\Sc_{r,+}^{n-1}}\langle \mathcal{C},\mu\rangle d\Sc_{r,+}^{n-1}+\int_{\Sc_{r}^{n-2}}g_{\alpha n}\vartheta^{\alpha}d\Sc_{r}^{n-2} 
+o(r'),\notag
\end{align}
where $\lim_{r'\to\infty} o(r')=0$.
Taking into account that $R\in L^1(M)$ and $H\in L^1(\Sigma)$, this clearly shows that the limit in the right-hand side of (\ref{massadm}) exists and is finite for any given asymptotically flat coordinate system. 

If we repeat the above computation using $\phi \mathcal{C}$ instead of $\mathcal{C}$, where $\phi$ is a cutoff function which equals $1$ in a neighborhood of infinity, then we easily see that the mass is continuous as a function on $\Mc_\tau$, for the fixed asymptotically flat chart. Since it is obviously affine, its smoothness follows at once. Thus, it remains to check that the mass does not depend on the asymptotically flat chart used to compute it. For this we need to show that $\mathbb R^n_+$ is rigid at 
infinity in a suitable sense. This uses harmonic coordinates as in \cite{Ba}.
\begin{proposition}
\label{exisharmcoord}
Suppose  $(M,g)$ is asymptotically flat with $g\in\Mc_\tau$, $\tau>(n-2)/2$, and $\{x_i\}$ are asymptotically flat coordinates defined on $M_{\infty}$. Then there exist smooth functions $\{x'_i\}$ on $M$ satisfying
\begin{equation}\notag
\begin{cases}
\Delta_gx'_{\a}=0\,&\text{in}\:M\,,
\\
\displaystyle\frac{\d x'_{\a}}{\d \eta}=0\,&\text{on}\:\Sigma\,,
\end{cases}
\end{equation}
for $\a=1,...,n-1$,
\begin{equation}\notag
\begin{cases}
\Delta_gx'_{n}=0\,&\text{in}\:M\,,
\\
x'_{n}=0\,&\text{on}\:\Sigma\,,
\end{cases}
\end{equation}
and
\begin{equation}\notag
\begin{cases}
x_i-x'_i\in C^{2,\alpha}_{-\tau+1}(M_{\infty})\,,&\text{if}\:n\geq 4\,;
\\
x_i-x'_i\in C^{2,\alpha}_{-\tau+1+\e}(M_{\infty})\,,&\text{if}\:n=3\,.
\end{cases}
\end{equation}
Moreover, the functions $\{x'_i\}$ form an asymptotic flat coordinate system in a neighborhood of infinity.
\end{proposition}
\begin{proof}
Let us first extend $x_i$ arbitrarily to smooth functions on $M$ satisfying $x_n=0$ on $\d M$.

If $n\geq 4$, then $-\tau+1$  is negative and the result follows from Proposition \ref{isomorphism}. Indeed, we use the fact that $\Delta_gx_i\in C^{0,\alpha}_{-\tau-1}(M)$ and $\partial x_{\alpha}/\partial\eta \in C^{1,\alpha}_{-\tau}(\Sigma)$, for $\a=1,..,n-1$, to solve for $z_i\in C^{2,\alpha}_{-\tau+1}(M)$ the equations 
$\Delta_g z_\alpha=\Delta_gx_\alpha$ and 
$\Delta_g z_n=\Delta_gx_n$ with boundary conditions $\partial z_\alpha/\partial\eta=\partial x_{\alpha}/\partial\eta$ and $z_n=0$, respectively. It is clear that $x'_i=x_i-z_i$ meets the conditions of the proposition.

If $n=3$, $-\tau+1$ may be positive, so we will need to make use of Propositions \ref{Sobolev:N}  and \ref{Sobolev:D} to find $x'_i$ as above. 
By Proposition \ref{Sobolev:lemma}, $\Delta_g x_i\in  C^{0,\alpha}_{-\tau-1}(M)\subset L^{q}_{0,-\tau-1+\e}(M)$ and $\d x_{\a}/\d \eta\in C^{1,\alpha}_{-\tau}(\Sigma)\subset L^{q}_{1-1/q,-\tau+1+\e}(\Sigma)$, for any $q>1$ and $\e>0$.
Assuming that $\e$ is chosen such that $\b=-\tau+1+\e\notin \mathbb{Z}$ and $\b>2-n$, it follows from Remark \ref{remsurj} that there exist $z_i\in L^{q}_{2,\b}(M)$ satisfying $\Delta_g z_\alpha=\Delta_gx_\alpha$ and 
$\Delta_g z_n=\Delta_g x_n$ with boundary conditions $\partial z_\alpha/\partial\eta=\partial x_\alpha/\partial\eta$ and $z_n=0$, respectively.
For $q>n$, Proposition \ref{Sobolev:lemma} implies that $z_i\in C^{1,\a}_{\b}(M)$, and Lemma \ref{lemma2}(a) ensures that $z_i\in C^{2,\a}_{-\tau+1+\e}(M)$. Finally, we set again $x'_i=x_i-z_i$.

That $\{x'_i\}$ form a coordinate system in some neighborhood of infinity follows from the fact that $|\nabla z_i|=O(r^{-\tau+\e})$.
\end{proof}

We now consider two asymptotically flat coordinate systems $\{x_i\}$ and $\{y_i\}$ on the same manifold and let  $\{x'_i\}$ and $\{y_i'\}$ be the corresponding harmonic coordinate systems as in Proposition \ref{exisharmcoord}. 
The following result describes the relationship between these coordinate systems. 
\begin{proposition}
\label{relatcoordharm}
If $(M,g)$ is as in Proposition \ref{exisharmcoord} then  there exists an orthogonal matrix $\{Q_i^j\}_{i,j=1}^{n}$  and constants $a_i$, $i=1,...,n$,  so that 
\[
x_i'=Q_{i}^jy_j'+a_i\,,
\]
with $Q^{n}_{\a}=Q^{\a}_{n}=a_n=0$, for $\a=1,...,n-1$.
\end{proposition}
\bp
We consider the double $(\widetilde{M},\widetilde{g})$ as in the proofs of Propositions \ref{Sobolev:N} and \ref{Sobolev:D}. Since $x'_i\in {\rm{ker}}\,\Delta_g\subset L^q_{2,\beta}(M)$ for all $1<\beta<2$ and $q\geq 1$, we can define functions $\widetilde{x}'_i\in {\rm{ker}}\, \Delta_{\widetilde{g}}\subset L^q_{2,\beta}(\widetilde{M})$ by $\widetilde{x}'_{\a}(x,j)=x'_{\a}(x)$ and $\widetilde{x}'_{n}(x,j)=(-1)^{j}x'_{n}(x)$ for $j=0,1$. We define $\widetilde{y}'_i$ in a similar way.  Although the coordinates $x_i$ and $y_i$ define different spaces $L^q_{k,\beta}(\widetilde{M})$ for $k\geq 1$, $ {\rm{ker}}\, \Delta_{\widetilde{g}}$ is independent of the chosen coordinates as observed in \cite[p. 676]{Ba}. 

Since ${\rm {dim}} ({\rm{ker}}\,\Delta_{\widetilde{g}})=n+1$ and the set $\{1,\tilde{y}'_1,...,\tilde{y}'_n\}$ is linearly independent, we can write $\widetilde{x}'_ i=Q^{j}_{i}\widetilde{y}'_j+a_i$ and then the result follows by using the boundary conditions on $x'_i$ and $y'_i$. We observe that $\{Q_{i}^{j}\}$ is orthogonal because the metric $g$ is asymptotially flat with respect to both $x'_i$ and $y'_i$.  
\ep

By eventually composing the coordinates with rigid motions of $\Rn$ we may assume that 
\begin{equation}\label{trans1}
\partial_{y_i}=\widetilde Q_{j}^i\partial_{x_j}, 
\end{equation}
with 
\begin{equation}\label{trans2}
\widetilde {Q}^{i}_{j}=\delta^{i}_{j}+O(r^{-\tau+\e}), \quad\text{and}\:\: \widetilde Q^\alpha_n=0\:\;\text{along}\: \Sigma. 
\end{equation} 
This is the promised rigidity at infinity of $\mathbb R^n_+$, which we now explore to complete the proof of Proposition \ref{massinv}.

As in \cite[p.680]{Ba} we write $R\star_g 1=d\tilde{\mathcal C}^{(x)}+{\mathcal D}$  in a given coordinate system $\{x_i\}$,  where   $\tilde{\mathcal C}^{(x)}=g^{ij}\omega_i^k\zeta_{kj}$, with $\nabla_g\partial_{x_i}=\omega_i^k\partial_{x_k}$ and $\zeta_{kj}=(\partial_{x_k}\wedge \partial_{x_j})\lrcorner\star_g 1$, and ${\mathcal D}=O(r^{-2\tau-2})$. 
Observe that the mass density $\mathcal C^{(x)}$ (defined in Proposition \ref{expasymreg}) in this same coordinate system satisties
$$
\lim_{r\to +\infty}\int_{\Sc_{r,+}^{n-1}}\langle \mathcal{C}^{(x)},\mu\rangle d\Sc_{r',+}^{n-1}
=\lim_{r\to +\infty}\int_{\Sc_{r,+}^{n-1}}\tilde{\mathcal C}^{(x)}.
$$
Proceeding as  in \cite[pp.681-682]{Ba} it follows from (\ref{trans1})-(\ref{trans2}) that
\[
\tilde{\mathcal C}^{(x)}-\tilde{\mathcal C}^{(y)}=d\left(\star_\delta\left(\widetilde Q^j_idx_j\wedge dx_i\right)\right)+O(r^{-2\tau-1+2\epsilon}),
\]
that is, the ADM densities differ by the sum of a total differential and a term that integrates to zero as $r\to +\infty$. This is the \lq simple but curious cancellation\rq\, mentioned in {\cite{Ba}}. Observing the chosen  orientation for $S^{n-2}$ in Definition \ref{defasymflat}, we find that 
\begin{align*}
\lim_{r\to +\infty}\int_{\Sc_{r,+}^{n-1}}\tilde{\mathcal C}^{(x)}  - 
\lim_{r\to +\infty}\int_{\Sc_{r,+}^{n-1}}\tilde{\mathcal C} ^{(y)}
& =
\lim_{r\to+\infty}\int_{\Sc^{n-2}_r}\star_\delta\left(\widetilde Q^j_idx_i\wedge dx_j\right)\\
& = 
\lim_{r\to+\infty}\int_{\Sc^{n-2}_r}\star_\delta\left(\widetilde Q^n_\alpha dx_{\a}\wedge dx_{n}\right).
\end{align*}
On the other hand, by (\ref{trans1})-(\ref{trans2}) we compute that 
\begin{eqnarray*}
g^{(y)}_{\alpha n} & = & \langle \partial_{y_{\alpha}},\partial_{y_n}\rangle=
\langle \widetilde Q^\alpha_i\partial_{x_i},\widetilde Q_j^n\partial_{x_j}\rangle \\
& = & \widetilde Q^\alpha_\beta\widetilde Q^n_j
g^{(x)}_{\beta j} =
\widetilde Q^n_jg^{(x)}_{\alpha j}+ O(r^{-2\tau+\epsilon})\\
& = & 
g_{\alpha n}^{(x)}+ \widetilde Q^n_\alpha +O(r^{-2\tau+\epsilon}).
\end{eqnarray*}
This completes the proof of Proposition \ref{massinv}.

\begin{remark}
\label{keyrem}
{\rm
In a coordinate system where $g_{\alpha n}=0$ along 
$\Sigma$ in the asymptotic region, 
the expression (\ref{massadm}) simplifies to 
\begin{equation}
\label{keyremeq}
{\mathfrak m}_{(M,g)}=\lim_{r\to +\infty}\int_{\Sc_{r,+}^{n-1}}(g_{ij,j}- g_{jj,i})\mu^i d\Sc_{r,+}^{n-1}. 
\end{equation}
In particular, this takes place if the metric is conformally flat near infinity, which is the case of the metric $\overline g$ constructed in Proposition \ref{conf:flat:assump} below. As a concrete example, 
consider the 
{\em half Schwarzschild space}, which is $M_m=\{x\in \mathbb R^n_+;|x|\geq (m/2)^{\frac{1}{n-2}}\}$ endowed with the conformal metric 
\[
g_m=\left(1+\frac{m}{2}|x|^{2-n}\right)^{\frac{4}{n-2}}\delta, \quad m>0.
\]
Thus, $g_m$ is scalar-flat with a  non-compact totally geodesic boundary given by $x_n=0$ and 
a straightforward computation using (\ref{massadm}) shows that 
\[
\label{masshsc}
\mathfrak m_{(M_m,g_m)}=(n-1)\omega_{n-1}m. 
\]
This means that $\mathfrak m_{(M_m,g_m)}$ is half the ADM mass  of the standard Schwarzschild space, which is the double of $(M_m,g_m)$ along its totally geodesic boundary.
A similar remark applies  for the mass invariants of the manifolds appearing in the doubling construction used in the proof of Theorem \ref{mainm} presented in the next  section.}
\end{remark}

\section{A proof of Theorem \ref{mainm}}
\label{firstproof}

In the following we denote the dependence of geometric invariants on the underlying metric by a subscript. 
In particular, we consider the conformal operators
\[
L_g=-a_n\Delta_g+R_g, \quad a_n=\frac{4(n-1)}{n-2},
\]
and 
\[
B_g= b_n\frac{\partial}{\partial\eta_g}+H_g, \quad b_n= \frac{2(n-1)}{n-2}, 
\]
acting on function on $M$ and $\Sigma$, respectively. 
We also recall  the function $r(x)$ defined in Section \ref{massgeo} as any smooth, positive extension of the asymptotic parameter $|x|$ to $M$.

The following result shows that, under the conditions of Theorem \ref{mainm}, the asymptotics of the metric $g$ can be substantially improved. This follows an idea first put forward by Schoen and Yau in their celebrated proof of the classical Positive Mass Theorem (\cite{SY2}).  

\begin{proposition}\label{conf:flat:assump}
Let $(M,g)$ be an asymptotically flat manifold with $g\in\Mc_{\tau}$, where $\tau>(n-2)/2$, and assume that $R_g\geq 0$ and  $H_g\geq 0$. Then  for any $\epsilon>0$ small enough there exists an asymptotically flat metric  $\bar{g}\in\Mc_{\tau-\epsilon}$ satisfying:\\ 
i) $R_{\bar{g}}\geq 0$ and $H_{\bar{g}}\geq 0$,  with $R_{\bar{g}}\equiv 0$ and  $H_{\bar{g}}\equiv 0$  near infinity;\\ 
ii) $\overline g$ is conformally flat near infinity;\\
iii) $|\mathfrak m_{(M,\bar g)}- \mathfrak m_{(M,g)}|\leq \e$. 
\end{proposition}
\begin{proof}
Our  argument uses the conformal method and proceeds similarly  to  the proof of \cite[Lemma 10.6]{LP}.
Let $\chi:\R\to [0,1]$ be a smooth cutoff function such that $\chi(t)=1$ for $t\leq 1$ and $\chi(t)=0$ for $t\geq 2$. 
For $R>0$ large define $\chi_R(x)=\chi(R^{-1}r(x))$ and set $g_R=\chi_R g+(1-\chi_R)\delta$. We are going to solve
\begin{equation}\label{sistema:uR}
\begin{cases}
L_{g_R}u_R=\chi_R R_g u_R\,,&\text{in}\:M\,,
\\
B_{g_R}u_R=\chi_R H_g u_R\,,&\text{on}\:\Sigma
\end{cases}
\end{equation}
for $u_R>0$ and $R$ large enough, and check that the conformal metric $\overline g=u_R^{\frac{4}{n-2}}g_R$ has all the desired properties.

We write $u_R=1+v_R$ and set 
$$
L_R=-\Delta_{g_R}+\frac{1}{a_n}\gamma_R, \quad
B_R=\frac{\d}{\d\eta_{g_R}}+\frac{1}{b_n}\bar{\gamma}_R,
$$
where $\gamma_R=R_{g_R}-\chi_R R_g$ and $\bar{\gamma}_R=H_{g_R}-\chi_R H_g$.  
Thus, (\ref{sistema:uR}) is equivalent to
\begin{equation}\label{sistema:psiR}
\begin{cases}
a_n L_Rv_R=-\gamma_R\,,&\text{in}\:M\,,
\\
b_nB_Rv_R=-\bar{\gamma}_R\,,&\text{on}\:\Sigma\,,
\end{cases}
\end{equation}
For any $\e>0$ we have $\|\gamma_R\|_{C^{0,\a}_{-\tau-2+\e}(M)}\to 0$ and  $\|\bar{\gamma}_R\|_{C^{1,\a}_{-\tau-1+\e}(\Sigma)}\to 0$ as $R\to \infty$. In what follows we are going to solve (\ref{sistema:psiR}) uniquely for $v_R\in C^{2,\a}_{-\tau+\e}(M)$, with $\|\psi_R\|_{C^{2,\a}_{-\tau+\e}(M)}\to 0$ as $R\to \infty$. 

Let us fix $0<\e<\tau-\frac{n-2}{2}$.  According to Proposition \ref{isomorphism}, the operator
$$
T:C^{2,\a}_{-\tau+\e}(M)\to C^{0,\a}_{-\tau+\e-2}(M)\times C^{1,\a}_{-\tau+\e-1}(\Sigma)
$$ 
given by $\displaystyle T u=(\Delta_gu,\frac{\d u}{\d\eta_g})$ is an isomorphism.

Let us consider $\displaystyle T_Ru=(L_Ru,B_Ru)$.
It follows from the easily established  estimates
$$
\|(\Delta_{g_R}-\Delta_g)u)\|_{ C^{0,\a}_{-\tau+\e-2}(M)}\leq \|g_R-g\|_{ C^{1,\a}_{0}(M)}\|u\|_{ C^{2,\a}_{-\tau+\e}(M)}\;,
$$
$$
\|(\d/\d\eta_{g_R}-\d/\d\eta_{g})u\|_{ C^{1,\a}_{-\tau+\e-1}(\Sigma)}\leq 
\|g_R-g\|_{ C^{0,\a}_{0}(M)}\|u\|_{ C^{2,\a}_{-\tau+\e}(M)}\;,
$$
$$
\|\gamma_R u\|_{ C^{0,\a}_{-\tau+\e-2}(M)}\leq 
\|\gamma_R\|_{ C^{0,\a}_{-2}(M)}\|u\|_{ C^{0,\a}_{-\tau+\e}(M)}\;,
$$
and
$$
\|\bar{\gamma}_R u\|_{ C^{1,\a}_{-\tau+\e-1}(\Sigma)}\leq 
\|\bar{\gamma}_R\|_{ C^{1,\a}_{-1}(\Sigma)}\|u\|_{ C^{1,\a}_{-\tau+\e}(M)}
$$
that $T_R-T$ is arbitrarily small in the operator norm as $R\to\infty$. From this we conclude that $T_R$ is also an isomorphism for large $R$, which provides a unique solution $v_R$ to (\ref{sistema:psiR}). 

Now we can choose $\bar{g}=g_R$ for $R$ large, proving (i) and (ii).
It is easy to prove that $g_R\to g$  in $\Mc_{\tau-\e}$, as $R\to\infty$, so that the property (iii) also holds.
\end{proof}

We will  now present our first proof of Theorem \ref{mainm}.
We first observe that in order to prove the inequality $\mathfrak m_{(M,g)}\geq 0$, it suffices to assume that $g$ satisfies the conclusion of Proposition \ref{conf:flat:assump}. Thus, we may assume that
$R_g\geq 0$ and $H_g\geq 0$ everywhere with $R_g=0$, $H_g= 0$ and $g$ conformally flat outside a compact set $K=\{x\in M\,;\:r(x)\leq C\}$.

Since $\Sigma$ is umbilic and $H_g=0$ outside $K$ it follows that  $\Sigma\backslash K$ is totally geodesic. This suggests to consider
the double $(\widetilde{M},\widetilde{g})$ of $(M,g)$ along $\Sigma$. More precisely, $\widetilde{M}=M\times \{0,1\}\slash\thicksim$, where $(y,0)\thicksim (y,1)$ for all $y\in\Sigma$, and $\widetilde{g}(y,j)=g(y)$ for all $y\in M$ and $j=0,1$.  
It is easy to check that 
$\widetilde{g}$ is $C^{2,\a}$ on $\widetilde{M}\backslash\widetilde{K}$,
where $\widetilde{K}$ is the double of $K$.
Observe that, if we 
consider the compact hypersurface (with boundary) $\Sigma_K=\Sigma\cap K$, we have that both  $\widetilde{g}|_{\widetilde{M}\backslash M}$ and on $\widetilde g_{M\backslash \Sigma}$  induce the same metric on $\Sigma$ and hence on $\Sigma_K$.
Also,
since $\Sigma_K$ has nonnegative mean curvature $H_g$ with respect to $M\backslash \Sigma$ and the unit normal $\eta$, it has nonpositive mean curvature $-H_g$ with respect to $\widetilde{M}\backslash M$ and the same unit normal vector $\eta$.

We can extend $\Sigma_K$ to a  closed hypersurface $\Sigma'$ in such a way that $\Sigma'\backslash\Sigma_K\subset \widetilde{M}\backslash\widetilde{K}$ and $\eta$ points to the unbounded connected component of $\widetilde{M}\backslash \Sigma'$; see Figure \ref{fig2}. Let $H^+$ be the mean curvature of $\Sigma'$ with respect to this unbounded component and $H^-$ be the one with respect to the bounded component, both calculated using a smooth extension of $\eta$ normal to $\Sigma'$. 
Observe that $H^-=-H^+=H_g\geq 0$ on $\Sigma_K$, the region where $\widetilde g$ is possibly nonsmooth. On the other hand, since $\widetilde{g}$ is $C^{2,\a}$ in $\widetilde{M}\backslash\widetilde{K}\supset \Sigma'\backslash \Sigma_K$, we see  that $H^+=H^-$ on $\Sigma'\backslash\Sigma_K$. Thus, $\widetilde{M}$ together with $\Sigma'$ satisfy the assumptions  of \cite[Theorem 1]{Mi}, which allows us  to infer that the ADM mass $m_{(\widetilde M,\widetilde g)}$ of the doubled manifold is nonnegative. As in Remark \ref{keyrem}, this mass is precisely $2\mathfrak m_{(M,g)}$, so we conclude that  $\mathfrak m_{(M,g)}\geq 0$, as desired.

\begin{figure}[!htb]
\centering
\includegraphics[scale=0.70]{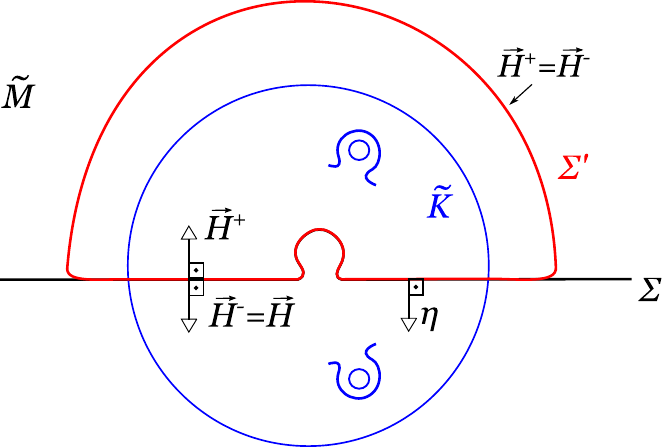}
\caption{The doubling construction.}
\label{fig2}
\end{figure}
\begin{remark} {\rm{Although originally stated in a  more restrictive setting, the result in \cite{Mi} also holds in our context where $\Sigma'$ may have finitely many connnected components and the unbounded region $N$ determined by this hypersurface  does not need to be diffeomorphic to $\R^n$ minus a ball. In fact, it is enough that $N$ is an asymptotically flat manifold in the classical sense. We can extend this result even further to the case where $\widetilde{M}\backslash N$ has a finite number of ends, allowing the generalization stated in Remark \ref{gener}.}} 
\end{remark}
We will now prove the rigidity statement in Theorem \ref{mainm}.
The key result is the lemma below, which says that a manifold $(M,g)$ as in Theorem \ref{mainm} which has minimal mass necessarily is Ricci-flat and has totally geodesic boundary. 
\begin{lemma}\label{scalar:rigidity}
Under the conditions of Theorem \ref{mainm}, if $\mathfrak m_{(M,g)}= 0$ then ${\rm Ric}_g=0$ and $A_g=0$.
\end{lemma}
\bp
As in the proof of \cite[Lemma 10.7]{LP}, we will use the variational characterization of the mass; see Proposition \ref{propvar}. 
We set $g_t=g+tk$, where $k$ is a compactly supported symmetric $2$-tensor on $M$. These metrics do not necessarily satisfy the positivity conditions on the scalar and mean curvatures but we remedy this as follows.
We consider the linear boundary value problem 
\begin{equation}\label{sistema:phi}
\begin{cases}
L_{g_t}u_t=R_gu_t\,,&\text{in}\:M\,,
\\
B_{g_t}u_t=H_gu_t\,,&\text{on}\:\Sigma\,,
\end{cases}
\end{equation}
for $t\in\mathbb R$.
If we write $u_t=1+v_t$ this is equivalent to
\begin{equation}\label{sistema:psi}
\begin{cases}
-a_n\Delta_{g_t}v_t+\gamma_t v_t=-\gamma_t\,,&\text{in}\:M\,,
\\
b_n\frac{\d v_t}{\d\eta_{g_t}}+\bar{\gamma_t} v_t=-\bar{\gamma_t}\,,&\text{on}\:\Sigma\,,
\end{cases}
\end{equation}
where $\gamma_t=R_{g_t}-R_g$ and $\bar{\gamma_t}=H_{g_t}-H_g$.  Since $\gamma_t$ and $\bar{\gamma_t}$ are both compactly supported and converge to zero in $C^k$ as $t\to 0$, we see that  $\|\gamma_t\|_{C^{0,\a}_{-\tau-2}(M)}\to 0$ and  $\|\bar{\gamma_t}\|_{C^{1,\a}_{-\tau-1}(\Sigma)}\to 0$ as $t\to 0$. Arguing as in the proof of Proposition \ref{conf:flat:assump}, for small $t$  we can solve (\ref{sistema:psi}) uniquely for $v_t\in C^{2,\a}_{-\tau}(M)$, with $\|v_t\|_{C^{2,\a}_{-\tau}(M)}\to 0$ as $t\to 0$. 
Thus, if we 
set $\widetilde{g}_t=u_t^{\frac{4}{n-2}}g_t$, it follows from (\ref{sistema:phi}) that
\begin{equation}\label{rel:conf}
R_{\widetilde{g}_t}=u_t^{-\frac{n+2}{n-2}}L_{g_t}u_t=
u_t^{-\frac{4}{n-2}}R_g\geq 0\,,
\:\:\:\:\:\:
H_{\widetilde{g}_t}=u_t^{-\frac{n}{n-2}}B_{g_t}u_t=u_t^{-\frac{2}{n-2}}H_g\geq 0\,.
\end{equation}

Notice that we have already proved that $\mathfrak m_{(M,g)}\geq 0$ for any $(M,g)$ as in Theorem \ref{mainm}. Hence,   the assumption $\mathfrak m_{(M,g)}=0$ means that $g$ is a minimum for the mass among the metrics $\widetilde{g}_t$, $|t|$ small.
On the other hand, if $\dot u=\frac{d}{dt}u_t|_{t=0}$ we easily see that 
$$
\frac{d}{d t}\Big|_{t=0}R_{\widetilde{g}_t}d M_{\widetilde{g}_t}=-\frac{4}{n-2}\dot{u}R_g d M_g+\frac{1}{2}R_g\left\langle\frac{4}{n-2}\dot{u}g+k,g\right\rangle d M_g,
$$
and
$$
\frac{d}{d t}\Big|_{t=0}2H_{\tilde{g}_t}d \Sigma_{\tilde{h}_t}=-\frac{4}{n-2}\dot{u}H_g d \Sigma_h+H_g\left\langle\frac{4}{n-2}\dot{u}g+k,g\right\rangle d \Sigma_h.
$$
Using Propositions \ref{propvar} and \ref{massinv}, and putting all these facts together, a straightforward computation gives 
\begin{eqnarray*}
0 & = & \frac{d}{dt}\Big|_{t=0}\mathfrak m_{(M,\widetilde g_t)}\\
&= & \int_M\left\langle k,{\rm Ric}_g\right\rangle d M_g+\int_{\Sigma}\left\langle k,A_g\right\rangle d \Sigma_h.
\end{eqnarray*}
Since this holds for any compactly supported $k$, the conclusion follows.
\ep

The proof of Theorem \ref{mainm} is completed using the following lemma. 
\begin{lemma}\label{ricci:rigidity} 
If $(M,g)$ as in Theorem \ref{mainm} satisfies ${\rm Ric}_g\geq 0$ and $A_g=0$ then it is  isometric to $\Rn$ with the standard flat metric.
\end{lemma}
\begin{proof}
The double $(\widetilde M,\widetilde g)$
of $(M,g)$ along its totally geodesic boundary $\Sigma$ is a complete $C^{2,\alpha}$ asymptotically flat manifold (with an empty boundary) satisfying ${\rm Ric}_{\widetilde{g}}\geq 0$, which  is well-known to be  isometric to $\mathbb R^n$ with the standard flat metric.
The result follows. 
\end{proof}

\section{Another proof of Theorem \ref{mainm}}
\label{anotherproof}

In this section we provide another proof of Theorem \ref{mainm} which is more in line with the classical arguments. Thus, in Subsection \ref{nleq7} we treat the case $n\leq 7$ and in Subsection \ref{setup} we treat the case when $M$ is spin. 
 
\subsection{The case $n\leq 7$}\label{nleq7}  
We present below another proof of Theorem \ref{mainm} for this case, following the ideas in \cite{SY1, SY2}.

Let us assume by contradiction that the mass ${\mathfrak m}_{(M,g)}$  is negative. We will construct an embedded asymptotically flat minimal hypersurface $\Hc\subset M$ and obtain a contradiction.  The boundary $\Sigma$ will work as a barrier when constructing $\Hc$.

By Proposition \ref{conf:flat:assump} we can assume that $(M,g)$ is asymptotically flat satisfying $g=u^{\frac{4}{n-2}}\delta$ and $R_g\equiv 0\equiv H_g$ near infinity. Hence, 
\begin{equation}\notag
\begin{cases}
\Delta u =0\,,&\text{in}\:\Rn\,,
\\
\displaystyle\frac{\d u}{\d x_n}=0\,,&\text{on}\:\d \Rn\,,
\end{cases}
\end{equation}
for $|x|$ large. Thus we can write 
\begin{equation}\label{exp:u}
u(x)=1+C|x|^{2-n}+O(|x|^{1-n})
\end{equation}
where $C=c(n){\mathfrak m}_{(M,g)}$ and $c(n)$ is a positive constant so that $C<0$ by our assumptions. (The proof of this asymptotic expansion is a simplified version of the arguments used for the function $v$ below. An alternative argument is a modification of the one in \cite{LP}, p.83.)

If $f_1$ and $f_2$ are positive functions decaying rapidly on $M$ and $\Sigma$ respectively, then we can find a solution to
\begin{equation}\label{eq:f1:f2}
\begin{cases}
L_{g}v=f_1\,,&\text{in}\:M\,,
\\
B_{g}v=f_2\,,&\text{on}\:\Sigma\,.
\end{cases}
\end{equation}
in the form $v(x)=1+\e|x|^{2-n}+O(|x|^{1-n})$, where we can make $\e>0$ arbitrarily small.  (See Section \ref{firstproof} for the definition of $L_g$ and $B_g$.)  In fact, as in the proof of Proposition \ref{conf:flat:assump} we rewrite the system (\ref{eq:f1:f2}) by making $v=1+\psi$, $\psi\in C^{2,\a}_{-\tau}(M)$, and obtain a new system in terms of $\psi$, which is solvable if $\widetilde{f}_1=f_1-R_g\in C^{2,\a}_{-\tau-2}(M)$ and $\widetilde{f}_2=f_2-H_g\in C^{1,\a}_{-\tau-1}(\Sigma)$. The solvability of this problem relies on the fact that $R_g,H_g\geq 0$, which ensures its uniqueness and allows us to use Proposition \ref{isomorphism}. Observe also that $\widetilde{f}_j$ has the same decay rate of $f_j$, $j=1,2$, because both $R_g$ and $H_g$ have compact support. 

In order to obtain the expansion for $v$, we rewrite the system for $\psi$ in terms of the background metric $u^{-\frac{4}{n-2}}g$ , which is Euclidean outside a compact set. Then we proceed as in the proof of Lemma \ref{lemma4} to estimate $|x|^{n-2}|\psi(x)|$ by a constant which can be assumed arbitrarily small by choosing $\widetilde{f}_1$ and $\widetilde{f}_2$  small with appropriate decay at infinity. The details are left to the reader.

In particular, in what follows we can assume that $R_g>0$ in $M$, $H_g>0$ in $\Sigma$, $g=u^{\frac{4}{n-2}}\delta$ in the end $M_{\infty}$, and $u$ has the asymptotic expansion (\ref{exp:u}) with $C<0$.

Let us consider the unit vector field $\nu=u^{-\frac{2}{n-2}}\frac{\d}{\d x_n}$ defined on $M_{\infty}$ and inward normal to $\Sigma\cap M_{\infty}$. We can extend $\nu$ to $M$ in such a way that it is still a unit vector on $M$ and still normal to $\Sigma$ outside $M_{\infty}$. 
As in \cite{schoen} we compute the divergence of $\nu$ to find that
$$
\text{div}_g \nu =-2(n-1)C\frac{x_n}{|x|^{n}}+O(|x|^{-n})\,,
$$ 
near infinity. So we can choose $a_0>0$ sufficiently large such that $\text{div}_g \nu (x)>0$ whenever $x_n\geq a_0$.

For $\s>0$ large and $a\geq 0$, we set
$$
\Gamma_{a,\s}=\{x=(\bar{x},x_n);\:|\bar x|=\s,\:x_n=a\}.
$$
Let $\Hc_{\s,a}\subset M$ be the $(n-1)$-hypersurface of least area having $\Gamma_{\s,a}$ as its boundary. This is possible because $H_g> 0$ along $\Sigma$.
We define
$$
A(\s)=\min\{\text{area}_g(\Hc_{\s,a});\:a\in[0,a_0]\}.
$$
We set $\Hc_{\s}=\Hc_{\s,a_{\s}}$ for $a_{\s}\in[0,a_0]$ such that  $\text{area}_g(\Hc_{\s,a})=A(\s)$.

For $R>0$ large, let us set $\Omega_{R}=\{x\in M;\:|x|<R\}$.
As in \cite{SY1}, we can choose $R_0$ large such that $x\mapsto |x|^2$ is a convex function for $|x|\geq R_0$. 
For each $\s$ large we choose $R_{\s}\geq R_0$ such that
$$
\{x=(\bar x,x_n)\in M_{\infty};\:|\bar x|=\s;\:x_n\leq a_0\}
\subset\Omega_{R_{\s}}.
$$
By the maximum principle we see that 
$\Hc_{\s}\subset \Omega_{R_{\s}}$.

Since we have changed the metric $g$ in such a way that $\Sigma$ has positive mean curvature, we have $\text{div}_g\nu<0$ along $\Sigma$. By continuity, one can choose $\e_{\s}>0$ small such that $\text{div}_g\nu(x)<0$ for any $x\in \Omega_{R_{\s}}$ such that $d_g(x,\Sigma)\leq \e_{\s}$.
In particular we can conclude that 
\begin{equation}\label{epsilon}
\Hc_{\s}\subset \{x\in \Omega_{R_{\s}};\:d_g(x,\Sigma)\geq \e_{\s}'\}.
\end{equation}
for some $0< \e_{\s}'< \e_{\s}$. This is done by means of the divergence theorem and the minimizing properties of $\Hc_{\sigma}$.
\begin{remark}{\rm{
If $M$ has more than one end, the same arguments apply observing that $|x|=R$, $R$ large, works as a barrier in the other ends as well, preventing $\Hc_{\s}$ from escaping to infinity.
}}
\end{remark}

Now the proof follows as in \cite{schoen}, observing that (\ref{epsilon}) ensures that the boundary does not interfere when calculating the area first and second variations formula for each $\Hc_{\s}$. The limiting area minimizing hypersurface $\Hc$, obtained as $\s\to\infty$, is asymptotically flat without boundary, and we will be able to change conformally the induced  metric on $\Hc$ to a scalar-flat metric. In dimensions $4\leq n\leq 7$, by analyzing the second variation formula of  area on $\Hc$ we conclude that its mass is negative, contradicting the classical version of the positive mass theorem. In dimension $n=3$, the contradiction is given by analyzing the integral of the Gaussian curvature of the surface $\Hc$  as well as its Euler number.

The rigidity statement follows as in the previous section.

\subsection{The case $M$ spin}\label{setup}
In this subsection we 
establish a Witten-type formula for the mass of  asymptotically flat spin manifolds in any dimension $n\geq 3$. As a consequence, we obtain another proof of Theorem \ref{mainm}  in the spin setting.

\begin{theorem}
\label{wittenform}
If $(M,g)$ is an asymptotically flat spin manifold of dimension $n\geq 3$ as in Theorem \ref{mainm} then \begin{equation}
\label{wittenform2}
\frac{1}{4}\mathfrak m_{(M,g)}=\int_M\left(|\nabla\psi|^2+\frac{R}{4}|\psi|^2\right)dM_g
+\frac{1}{2}\int_{\Sigma}H|\psi|^2d\Sigma_h,
\end{equation} 
where $\psi$ is a suitable nontrivial harmonic spinor globally defined on $M$.
\end{theorem}

We start by recalling some basic facts on spinors in manifolds with boundary. The reader will find a more detailed account of
this preparatory material in \cite{F, HMZ, HMR}.

\subsubsection{The integral Lichnerowicz formula on spin manifolds with boundary}\label{integ}
We consider a spin manifold $\Omega$ of dimension $n\geq 3$ endowed with a Riemannian metric $g$. We denote by $\mathbb S\Omega$ the spin bundle of $\Omega$ and by $\nabla$ both the Levi-Civita connection of $T\Omega$
and its lift to $\mathbb S\Omega$. The corresponding Dirac operator $D:\Gamma(\mathbb S\Omega)\to\Gamma(\mathbb S\Omega)$ is locally given by
\begin{equation}
\label{defdirac}
D\psi=\sum_{i=1}^n \gamma(\eh_i)\nabla_{\eh_i}\psi,\quad \psi\in\Gamma(\mathbb SM),
\end{equation}
where $\{\eh_i\}_{i=1}^n$ is a local orthonormal frame and $\gamma:T\Omega\times \mathbb S \Omega\to\mathbb S\Omega$ is the Clifford product.

If $\Sc$ is the boundary of $\Omega$, which we assume oriented by its inner unit normal $\nu=-\eta$, then, given a spinor $\psi\in\Gamma(\mathbb S\Omega)$, a well-known computation gives the integral Lichnerowicz formula:
\begin{equation}\label{parts}
\int_\Omega\left(|\nabla \varphi|^2-|D\varphi|^2+\frac{R}{4}|\varphi|^2\right)d\Omega=-\int_{\Sc}
\left\langle B_\nu\varphi,\varphi\right\rangle d\Sc,
\end{equation}
where 
\[
B_X=\gamma(X) D+ \nabla_X,\quad X\in \Gamma(TM).
\]

For our purposes it will be convenient to rewrite the left-hand side above in terms of the mean curvature $H$ of $\Sc$.
We first note that $\Sc$  carries the spin bundle $\mathbb S\Omega|_\Sc$, obtained by restricting  $\mathbb S\Omega$ to $\Sc$. This becomes a Dirac bundle if its Clifford product is
$$
\gamma^{\intercal}(X)\varphi=\gamma(X)\gamma(\nu) \varphi,\quad X\in \Gamma(T\Sc), \quad \varphi\in \Gamma(\mathbb S\Omega|_\Sc),
$$
and its connection is
\begin{equation}\label{conn0}
\nabla^{\intercal}_X\varphi  =  \nabla_X\varphi-\frac{1}{2}\gamma^{\intercal}(AX)\varphi,
\end{equation}
where $A$ is the shape operator of $\Sc$.
The corresponding Dirac operator $D^{\intercal}:\Gamma(\mathbb S\Omega|_\Sc)\to\Gamma(\mathbb S\Omega|_\Sc)$ is
$$
D^{\intercal}\varphi=\sum_{j=1}^{n-1}\gamma^{\intercal}(f_j)\nabla^{\intercal}_{f_j}\varphi.
$$
A well-known computation shows that 
\begin{equation}\label{par}
D^{\intercal}\varphi=\frac{H}{2}\varphi-B_\nu\varphi,
\end{equation}
so that (\ref{parts}) is equivalent to
\begin{equation}\label{parts2}
\int_\Omega\left(|\nabla \varphi|^2-|D\varphi|^2+\frac{R_g}{4}|\varphi|^2\right)d\Omega=\int_{\Sc}
\left\langle D^{\intercal}\varphi-\frac{H}{2}\varphi,\varphi\right\rangle d\Sc.
\end{equation}

\subsubsection{A boundary value problem for spinors}\label{bvpspin}
Here we follow \cite{GN} and discuss  a certain boundary value problem for spinors on an asymptotically flat manifold with a non-compact boundary $\Sigma$. We start by observing that if $\nu$ is the inward unit normal to $\Sigma$ then the linear map $\varepsilon=i\gamma(\nu):\mathbb SM|_\Sigma\to\mathbb SM|_\Sigma$ is a self-adjoint involution. Thus, we have a (pointwise) decomposition,
\begin{equation}\label{ortho}
\mathbb SM|_\Sigma= V_+\oplus V_-,
\end{equation}
corresponding to the eigenbundles of $\varepsilon$, that is, $V_\pm=\{\varphi\in \mathbb SM|_\Sigma;\varepsilon \varphi=\pm\varphi\}$. We denote the corresponding projections by
$P_\pm:\mathbb SM|_\Sigma\to V_{\pm}$,
\[
P_\pm=\frac{1}{2}\left({\rm Id}\pm \varepsilon\right),
\]
and we set $\varphi=\varphi_++\varphi_-$, $\varphi_{\pm}=P_\pm\varphi$.
It is easy to check that $D^\intercal P_\pm=P_\mp D^\intercal$.

\begin{proposition}\label{nice}
If $\varphi\in\Gamma(V_\pm)$ then $\langle D^\intercal \varphi,\varphi\rangle=0$.
\end{proposition}

\begin{proof}
We compute
\[
\langle D^\intercal \varphi,\varphi\rangle=\langle D^\intercal P_{\pm}\varphi,P_\pm\varphi\rangle=\langle P_\mp D^\intercal \varphi,P_\pm\varphi\rangle=0,
\]
where in the last step we used that the decomposition (\ref{ortho}) is orthogonal.
\end{proof}

It turns out that the projections $P_\pm$ define nice boundary conditions for the Dirac operator $D$ of $M$. More precisely, 
the following result holds.

\begin{proposition}\label{bdp}
If $(M,g)$ is as in Theorem \ref{mainm} and $\phi\in\Gamma(\mathbb SM)$ satisfies $\nabla\phi\in L^2(\mathbb SM)$, then there exists a unique $\xi\in L^2_1(\mathbb SM)$
solving the boundary value problem
\[
\left\{
\begin{array}{ccccc}
D\xi & = & -D\phi & {\rm in} & M\\
\xi_- & = & 0 & {\rm on} & \Sigma
\end{array}
\right.
\]
\end{proposition}

\begin{proof}
The assumption $\nabla \phi\in L^2(\mathbb SM)$
implies, via (\ref{defdirac}) and Cauchy-Schwarz, that  
$D \phi\in L^2(\mathbb SM)$. The result is then an immediate consequence of   
\cite[Corollary 3.16]{GN}.
\end{proof}
  
\subsubsection{The proof  of the Witten-type mass formula}\label{formmass}

In this subsection we prove Theorem \ref{wittenform} by showing that the mass formula (\ref{wittenform2}) holds true. 
As already mentioned above, our proof adapts Witten's well-known argument as reported in \cite{LP}
to  the class of asymptotically flat manifolds we consider here. 

We first claim that, starting from an arbitrary asymptotically flat coordinate system $\{x_i\}$, we can always produce another such coordinate system $\{x_i'\}$ such that $g_{\alpha n}=0$ along $\Sigma$ near infinity; see Remark \ref{keyrem}.   
In effect, it follows from (\ref{unitnormalexp}) that 
$\nabla^{l}(\eta+\partial_{x_n})=O(r^{-\tau-l})$ for $l=0,1,2$. Also, from (\ref{qual}) we see that $A=O(r^{-\tau-1})$ and $\nabla A=O(r^{-\tau-2})$. Thus, the claim is verified if we choose $\{x_i'\}$ so that $x_\alpha'=x_\alpha$ and $\partial_{x_n'}=-\eta$ along $\Sigma$ and extend this to the whole asymptotic region in the obvious manner. In such a coordinate system, we can use the simplified expression (\ref{keyremeq}) to compute $\mathfrak m_{(M,g)}$.  

With this preliminary remark at hand, we start the proof by fixing a {\em constant} spinor $\phi$ with respect to the given asymptotically flat chart, which means that $\partial_i\phi=0$ in the asymptotic region. Moreover, we may assume that $|{\phi}|\to 1$ and $\phi_-=0$ along the boundary of this region. We extend $\phi$ as zero to the rest of $\Sigma$, so that $\phi_-=0$ everywhere, and finally we extend $\phi$ to the rest of $M$ in an arbitrary manner. 
The well-known formula for the spin connection shows that  
$\nabla \phi\in L^2(\mathbb SM)$,  so we may find $\xi\in L^2_1(\mathbb SM)$ as in Proposition \ref{bdp}. It is immediate that 
\begin{equation}
\label{decompo}
\psi=\xi+\phi
\end{equation}  
satisfies 
\[
\left\{
\begin{array}{ccccc}
D\psi & = & 0 & {\rm in} & M\\
\psi_- & = & 0 & {\rm on} & \Sigma
\end{array}
\right.
\]

We now apply (\ref{parts})-(\ref{parts2}) in the usual way to the region $M_r$ with boundary  $\Sigma_r\cup \Sc_{r,+}^{n-1}$  to obtain
\begin{eqnarray*}
\int_{M_r}\left(|\nabla \psi|^2+\frac{R}{4}|\psi|^2\right)dM 
& = & \int_{\Sigma_r}\left\langle D^{\intercal}\psi,\psi\right\rangle d\Sigma_r-
 \frac{1}{2} \int_{\Sigma_r}H|\psi|^2 d\Sigma_r\\
& & \quad +
      \Re\int_{\Sc_{r,+}^{n-1}}\left\langle B_\nu\psi,\psi\right\rangle d\Sc_{r,+}^{n-1},
\end{eqnarray*}
where $\eta=-\nu$ is the outward unit normal and $\Re$ denotes  real part. 
The boundary condition $\psi_-=0$ implies, via Proposition \ref{nice}, that the first integral in the right-hand side vanishes. 
By sending $r\to +\infty$ we get 
\begin{eqnarray*}
\int_{M}\left(|\nabla \psi|^2+\frac{R}{4}|\psi|^2\right)dM 
& = & -\frac{1}{2}\int_{\Sigma}H|\psi|^2 d\Sigma_r\\
& & \quad +
      \lim_{r\to +\infty}\Re\int_{\Sc_{r,+}^{n-1}}\left\langle \nabla_\eta\psi,\psi\right\rangle d\Sc_{r,+}^{n-1},
\end{eqnarray*}
so we must check that 
\begin{equation}
\label{checkf}
\lim_{r\to +\infty}\Re\int_{\Sc_{r,+}^{n-1}}\left\langle \nabla_\eta\psi,\psi\right\rangle d\Sc_{r,+}^{n-1}=\frac{1}{4}\mathfrak m_{(M,g)}.
\end{equation}

We use (\ref{decompo}) to split the integral as
\begin{eqnarray*}
\Re\int_{\Sc_{r,+}^{n-1}}\left\langle \nabla_\eta\psi,\psi\right\rangle d\Sc_{r,+}^{n-1} & = & \Re\int_{\Sc_{r,+}^{n-1}}\left\langle \nabla_\eta\phi,\phi\right\rangle  d\Sc_{r,+}^{n-1}+\Re\int_{\Sc_{r,+}^{n-1}}\left\langle \nabla_\eta\phi,\xi\right\rangle d\Sc_{r,+}^{n-1}\\
& & \quad + \Re\int_{\Sc_{r,+}^{n-1}}\left\langle \nabla_\eta\xi,\xi\right\rangle d\Sc_{r,+}^{n-1} +\Re\int_{\Sc_{r,+}^{n-1}}\left\langle \nabla_\eta\xi,\phi\right\rangle d\Sc_{r,+}^{n-1}.
\end{eqnarray*}
As explained in \cite{LP}, algebraic cancellations and the decay properties of $\nabla\phi$ and $\xi$ imply that the first three terms eventually vanish at infinity, so we must evaluate the fourth one as $r\to +\infty$. 

In order to handle this  limit  we note that asymptotic flatness means that 
\[
g_{ij}=\delta_{ij}+a_{ij}, \quad a_{ij}=O(r^{-\tau}). 
\] 
Hence, the coordinate frame $\partial_i$ can be orthonormalized to yield 
\[
\eh_i=\partial_i-\frac{1}{2}a_{ij}\partial_j+O(r^{-\tau}),
\]
 which gives 
\begin{equation}\label{commut}
\eh_i\cdot\eh_j\cdot=\partial_i\cdot\partial_j\cdot +O(r^{-\tau}),
\end{equation}
where from now on we represent Clifford product by a dot. Following \cite{LP} we introduce the $(n-2)$-form
\[
\omega=\left\langle[\eh_l\cdot,\eh_m\cdot]\phi,\xi\right\rangle \eh_l\lrcorner\eh_m\lrcorner dM.
\]
A straightforward computation gives 
\[
d\omega=-4\left(\langle B_{\eh_l}\phi,\xi\rangle-\langle \phi,B_{\eh_l}\xi\rangle\right){\eh_l}\lrcorner dM,
\]
where
\begin{equation}\label{lem3}
B_{\eh_l}=\nabla_{\eh_l}+\eh_l\cdot D=(\delta_{lm}+\eh_l\cdot\eh_m\cdot)\nabla_{\eh_m}=\frac{1}{2}[\eh_l\cdot,\eh_m\cdot]\nabla_{\eh_m}.
\end{equation}
In particular,
\begin{equation}\label{lem4}
\int_{\Sc_{r,+}^{n-1}}\langle B_{\eh_l}\xi,\phi\rangle {\eh_l}\lrcorner dM - \int_{\Sc_{r,+}^{n-1}}\langle B_{\eh_l}\phi,\xi\rangle {\eh_l}\lrcorner dM= \frac{1}{4}\int_{\Sc_{r}^{n-2}}\omega.
\end{equation}

In Witten's original argument, the boundary term in the right-hand side of (\ref{lem4}) vanishes because the integration in the left-hand side  is performed over a closed sphere. In our case, $\Sc_{r,+}^{n-1}$ is a hemisphere and this terms contributes with an integral over $\Sc_r^{n-2}=\partial \Sc_{r,+}^{n-1}$ which, as we shall see, vanishes at infinity. 
To see this we integrate by parts to get
\begin{eqnarray*}
\Re\int_{\Sc_{r,+}^{n-1}}\left\langle \nabla_{\eh_l}\xi,\phi\right\rangle{\eh_l}\lrcorner dM & = & \Re\int_{\Sc_{r,+}^{n-1}}\left\langle (B_{\eh_l}-\eh_l\cdot D)\xi,\phi\right\rangle{\eh_l}\lrcorner dM\\
& = & \Re\int_{\Sc_{r,+}^{n-1}}\left\langle B_{\eh_l}\phi,\xi\right\rangle{\eh_l}\lrcorner dM + \frac{1}{4}\Re\int_{\Sc_r^{n-2}}\omega\\
 & & \quad - \Re\int_{\Sc_{r,+}^{n-1}}\left\langle \eh_l\cdot D\xi,\phi\right\rangle{\eh_l}\lrcorner dM\\
 & = & \Re\int_{\Sc_{r,+}^{n-1}}\left\langle B_{\eh_l}\phi,\xi\right\rangle{\eh_l}\lrcorner dM + \frac{1}{4}\Re\int_{\Sc_r^{n-2}}\omega\\
 & & \quad  +\Re\int_{\Sc_{r,+}^{n-1}}\left\langle \eh_l\cdot D\phi,\phi\right\rangle{\eh_l}\lrcorner dM,
\end{eqnarray*}
where in the last step we used (\ref{decompo}) and the fact that $\psi$ is harmonic. Again due to the decay properties, the first integral in the right-hand side above vanishes at infinity. Also, the standard computation as in \cite{LP} shows that
\[
\lim_{r\to +\infty}\Re\int_{\Sc_{r,+}^{n-1}}\left\langle \eh_l\cdot D\phi,\phi\right\rangle{\eh_l}\lrcorner dM=\lim_{r\to +\infty}\frac{1}{4}\int_{\Sc_{r,+}^{n-1}}(g_{ij,j}-g_{jj,i})\mu^i d\Sc_{r,+}^{n-1}.
\]
Thus, 
it remains to check that
\begin{equation}\label{checkform}
\lim_{r\to +\infty}\Re\int_{\Sc_r^{n-2}}\omega=
0.
\end{equation}

Using (\ref{commut}) and 
restricting to $\Sigma$ in the asymptotic region we have
\begin{equation}\label{intzero}
\omega= 4\langle \partial_\alpha\cdot\partial_n\cdot \phi,\xi\rangle\partial_\alpha\lrcorner\partial_n dM + \langle O(r^{-\tau})\cdot \phi,\xi\rangle \partial_\alpha\lrcorner\partial_n dM.
\end{equation}
Since $\partial_\alpha\partial_n\lrcorner dM=d\Sc_r^{n-2}=O(r^{n-2})$ we see that the last term in the right-hand side integrates to zero at infinity. 
On the other hand, 
if $\dag$ means transpose conjugation then 
$
(\eh_\alpha\cdot \eh_n\cdot)^\dag=\eh_n\cdot\eh_\alpha\cdot
$
and hence 
$
(\partial_\alpha\cdot \partial_n\cdot)^\dag=\partial_n\cdot\partial_\alpha\cdot
+O(r^{-\tau})
$ by (\ref{commut}).
By using Clifford relations in the asymptotic region we get
\begin{eqnarray*}
0 & = & -2g_{\alpha n}\\
  & = & \partial_\alpha\cdot \partial_n\cdot+ \partial_n\cdot\partial_\alpha\cdot\\
  & = & \partial_\alpha\cdot \partial_n\cdot+(\partial_\alpha\cdot \partial_n\cdot)^\dag+ O(r^{-\tau}),
\end{eqnarray*}
which gives $\partial_\alpha\cdot \partial_n\cdot=O(r^{-\tau})$.  
Thus, the first term in the right-hand side of (\ref{intzero}) also integrates to zero as $r\to+\infty$.  
This completes the proof of Theorem  \ref{wittenform}.

With Theorem \ref{wittenform} at hand, we can easily produce a proof of Theorem \ref{mainm} in the spin case. First, it is immediate from (\ref{wittenform2}) that  a spin manifold  as in Theorem \ref{mainm} satisfies the mass inequality $\mathfrak m_{(M,g)}\geq 0$. Moreover, if $\mathfrak m_{(m,g)}=0$ then $(M,g)$ carries a non-trivial parallel spinor, say $\psi$. In particular, $g$ is Ricci flat. Also, since $i\nu\cdot \psi=\psi$ along $\Sigma$, we see after differentiation that $ AX\cdot \psi=0$ for any $X$ tangent to $\Sigma$. 
Since $(M,g)$ actually carries as many parallel spinors as the model space $(\mathbb R^n_+,\delta)$, we conclude that $\Sigma$ is totally geodesic and the rigidity statement follows from Lemma \ref{ricci:rigidity}.

\setcounter{equation}{0}
\appendix\section{\\ The proof of Proposition  \ref{isomorphism}}
\label{prooftech1}

In this technical appendix  we present a proof of Proposition \ref{isomorphism}. The argument follows from a series of lemmas which taken together establish the  mapping properties of the operator $T$ appearing in the proposition. Our proof is inspired by the ideas in \cite{CSCB}, where the case of manifolds without boundary is treated.     

\begin{lemma}\label{lemma1}
If $2-n<\gamma< 0$ there exists $C=C(n)>0$ such that, for all $u\in C^{2}_{\gamma}(\Rn)$, we have 
$$
\|u\|_{C^{0}_{\gamma}(\Rn)}\leq C\|\Delta u\|_{C^{0}_{\gamma-2}(\Rn)}+C\left\|\frac{\d u}{\d x_n}\right\|_{C^{0}_{\gamma-1}(\d\Rn)}.
$$
\end{lemma}
\bp
We set $\phi(x,y)=|x-y|^{2-n}+|x-\widetilde{y}|^{2-n}$, where $\widetilde{y}=(y_1,...,y_{n-1}, -y_n)$ if $y=(y_1,...,y_n)$.
Observe that $\Delta_x\phi(x,y)= 0$ for any $x,y\in \Rn$, and $\frac{\d}{\d x_n}\phi(x,y)=0$ for any $x\in\d \Rn$ and $y\in \Rn$. Then, for any $y\in \Rn$ with $|y|<R$ , Green's formula yields
\begin{eqnarray*}
(n-2)\omega_{n-1}u(y)& = &
-\int_{x\in \Rn, |x|\leq R}\phi(x,y)\Delta u(x)dx\\
& & \quad -\int_{x\in \d\Rn, |x|\leq R}\phi(x,y)\frac{\d u}{\d x_n}(x)d\sigma_R(x)
\\
 & & \quad \quad +\int_{x\in \Rn, |x|= R}\phi(x,y)\frac{\d u}{\d r}(x)d\sigma_R(x)\\
 & & \quad\quad\quad  -\int_{x\in \Rn, |x|= R}\frac{\d \phi}{\d r}(x,y)u(x)d\sigma_R(x).
\end{eqnarray*}
Choosing $R\geq 2|y|$ and using the fact that $u\in C^{2}_{\gamma}(\Rn)$, one can check that
\[
\int_{x\in \Rn, |x|= R}\phi(x,y)\left|\frac{\d u}{\d r}(x)\right|d\sigma_R(x)
\leq C(n)R^{\gamma}.
\]
Also, since 
$$
\left|\frac{\d}{\d r}|x-y|^{2-n}\right|
\leq C(n)R^{1-n}\,,
$$
for any  $u\in \R^n$ with  $R\geq 2|y|$, we get 
$$
\int_{x\in \Rn, |x|= R}\left|\frac{\d\phi}{\d r}(x,y)u(x)\right|d\sigma_R(x)\leq C(n)R^{\gamma}.
$$
Hence, taking the limit as $R\to\infty$ in Green's formula above and using the hypothesis $\gamma<0$ we obtain
\begin{equation}\label{lemma1:2}
(n-2)\omega_{n-1}u(y)=-\int_{x\in \Rn}\phi(x,y)\Delta u(x)dx-\int_{x\in \d\Rn}\phi(x,y)\frac{\d u}{\d x_n}(x)d\sigma(x).
\end{equation}
Since $2-n<\gamma$, we can use the fact that 
$$
\int_{x\in \Rn}|x-y|^{2-n}|x|^{\gamma-2}dx+\int_{x\in\d \Rn}|x-y|^{2-n}|x|^{\gamma-1}d\sigma(x)\leq C(n)|y|^{\gamma}
$$
for any $y\in \R^n$, so that it follows from  (\ref{lemma1:2}) that
\begin{eqnarray*}
|y|^{-\gamma}|u(y)|
&\leq  & 
C\int_{x\in \Rn}|y|^{-\gamma}\phi(x,y)\Delta u(x)|dx\notag
\\
& & \quad +C\int_{x\in \d\Rn}|y|^{-\gamma}\phi(x,y)\left|\frac{\d u}{\d x_n}(x)\right|d\sigma(x),
\\
&\leq & C\int_{x\in \Rn}|y|^{-\gamma}\phi(x,y)|x|^{\gamma-2}\|\Delta u(x)\|_{C^{0}_{\gamma-2}(\Rn)}dx
\\
& & \quad + C\int_{x\in \d\Rn}|y|^{-\gamma}\phi(x,y)|x|^{\gamma-1}\left\|\frac{\d u}{\d x_n}(x)\right\|_{C^{0}_{\gamma-1}(\d\Rn)}d\sigma(x)
\\
&\leq & C\|\Delta u(x)\|_{C^{0}_{\gamma-2}(\Rn)}+C\left\|\frac{\d u}{\d x_n}(x)\right\|_{C^{0}_{\gamma-1}(\d\Rn)},
\end{eqnarray*}
which proves the lemma.
\ep

\begin{lemma}\label{lemma2}
Let $(M,g)$ be an asymptotically flat manifold with $g\in\Mc_\tau$, $\tau>0$, and boundary $\Sigma$, and  let $\gamma\in\R$. Then the following assertions hold:
\\(a) There exists $C=C(M, g, \gamma)>0$ such that, if $u\in C^{1,\a}_{\gamma}(M)$, $\Delta_g u\in C^{0,\a}_{\gamma-2}(M)$ and $\d u/\d \eta_g\in C^{1,\a}_{\gamma-1}(\Sigma)$, then $u\in C^{2,\a}_{\gamma}(M)$ and we have
$$
\|u\|_{C^{2,\a}_{\gamma}(M)}\leq C\|\Delta_g u\|_{C^{0,\a}_{\gamma-2}(M)}+C\left\|\frac{\d u}{\d \eta_g}\right\|_{C^{1,\a}_{\gamma-1}(\Sigma)}
+C\|u\|_{C^{0}_{\gamma}(M)}.
$$
\\(b) Assume that $g=\delta$ outside a compact set and $2-n<\gamma<0$.  Then there exists $C=C(M, g, \gamma)>0$ and a compact set $K\subset M$  such that, for all $u\in C^{2}_{\gamma}(M)$,
$$
\|u\|_{C^0_{\gamma}(M)}
\leq 
C\|\Delta_g u\|_{C^0_{\gamma-2}(M)}+C\left\|\frac{\d u}{\d\eta_g}\right\|_{C^0_{\gamma-1}(\Sigma)}
+C\|u\|_{C^1(K)}.
$$
In particular, if $g=\delta$ outside a compact set and $2-n<\gamma<0$, then there exists $C=C(M, g, \gamma)>0$ and a compact set $K\subset M$  such that,  if $u\in C^{0,\a}_{\gamma}(M)$, then $u\in C^{2,\a}_{\gamma}(M)$ and we have
$$
\|u\|_{C^{2,\a}_{\gamma}(M)}\leq C\|\Delta_g u\|_{C^{0,\a}_{\gamma-2}(M)}+C\left\|\frac{\d u}{\d \eta_g}\right\|_{C^{1,\a}_{\gamma-1}(\Sigma)}
+C\|u\|_{C^{1}(K)}.
$$
\end{lemma}
\bp
We can identify the end  $M_{\infty}$ with $\Rn\backslash\{x\in\Rn, \:|x|>1\}$ under the given asymptotically flat chart.  For $R\geq 1$ we will denote by $K_R$ the compact set $M\backslash \{x\in M_{\infty};\:|x|>R\}$. 
Finally, for any subset $\Omega\subset M$, we define $\d'\Omega=\Omega\cap \Sigma$.


For the proof of item (a), let us set $A=\{x\in M_{\infty};\:1<|x|< 4\}$ and $\widetilde{A}=\{x\in M_{\infty};\:2<|x|< 3\}$. 
For $R\geq 1$ we also set $A_R=\{x\in M_{\infty};\:R<|x|<4R\}$ and $\widetilde{A}=\{x\in M_{\infty};\:2R<|x|< 3R\}$ so that $A_1=A$ and $\widetilde{A}_1=\widetilde{A}$. 
Let $0\leq \chi\leq 1$ be a smooth cutoff function satisfying $\chi\equiv 1$ in $\widetilde{A}$ and $\chi\equiv 0$ in $M\backslash A$, and let $u\in C^{1,\a}_{\gamma}(M)$.  We set
$u_R(x)=u(Rx)$ for $x\in A$ and define a metric $g_R$ on $A$ by $(g_R)_{ij}(x)=g_{ij}(Rx)$.

It follows from elliptic regularity that $u\in C^{2,\a}_{loc}(M)$ and 
\begin{equation}\label{lemma2:1}
\|\chi u_R\|_{C^{2,\a}(A)}\leq C\|\Delta_{g_R} (\chi u_R)\|_{C^{0,\a}(A)}+C\left\|\frac{\d}{\d \eta_{g_R}}(\chi u_R)\right\|_{C^{1,\a}(\partial'A)},
\end{equation}
for some $C=C(M,g)$.
Observe that  
\begin{equation}\label{lemma2:2}
 \|u_R\|_{C^{2,\a}(\widetilde{A})}\leq \|\chi u_R\|_{C^{2,\a}(A)}\,,
\end{equation}
\begin{equation}\label{lemma2:3}
\|\Delta_{g_R}(\chi u_R)\|_{C^{0,\a}(A)}\leq 
C\|\Delta_{g_R} u_R\|_{C^{0,\a}(A)}+C\|u_R\|_{C^{1,\a}(A)}
\end{equation}
and 
\begin{equation}\label{lemma2:4}
\left\|\frac{\d}{\d\eta_{g_R}}(\chi u_R)\right\|_{C^{1,\a}(\d 'A)}\leq 
C\left\|\frac{\d u_R}{\d\eta_{g_R}} \right\|_{C^{1,\a}(\d 'A)}+C\|u_R\|_{C^{1,\a}(\d 'A)}\,,
\end{equation}
so that 
\begin{equation}\label{lemma2:4'}
\|u_R\|_{C^{2,\a}(\widetilde{A})}
\leq
C\|\Delta_{g_R} u_R\|_{C^{0,\a}(A)}+C\left\|\frac{\d u_R}{\d\eta_{g_R}} \right\|_{C^{1,\a}(\d 'A)}+C\|u_R\|_{C^{1,\a}(A)}.
\end{equation}
Expanding this in terms of  $C^0$ norms,    
multiplying by $R^{-\gamma}$ and rewriting the result in terms of $u$ and $g$, we get 
\begin{equation}\label{lemma2:4''}
\|u\|_{C^{2,\a}_{\gamma}(\widetilde{A}_R)}
\leq  C\|\Delta_g u\|_{C^{0,\a}_{\gamma-2}(A_R)}+C\left\|\frac{\d u}{\d\eta_g}\right\|_{C^{1,\a}_{\gamma-1}(\d 'A_R)}
+C\|u\|_{C^{1,\a}_{\gamma}(A_R)}.
\end{equation}
Since this holds for R arbitrarily large, we conclude that $u\in C^{2,\a}_{\gamma}(M)$. 

Combining (\ref{lemma2:4'}) with a well-known interpolation inequality, namely,
\[
\|v\|_{C^{1,\a}(A)}\leq \e\|v\|_{C^{2,\a}(A)}+C(\e)\|v\|_{C^{0}(A)},
\]
and procceding as in (\ref{lemma2:4''}) we obtain
\begin{eqnarray*}
\|u\|_{C^{2,\a}_{\gamma}(\widetilde{A}_R)}
&\leq & C\|\Delta_g u\|_{C^{0,\a}_{\gamma-2}(A_R)}+C\left\|\frac{\d u}{\d\eta_g}\right\|_{C^{1,\a}_{\gamma-1}(\d 'A_R)}
\\
& & \quad +C(\e)\|u\|_{C^{0}_{\gamma}(A_R)}+\e\|u\|_{C^{2,\a}_{\gamma}(A_R)}.
\end{eqnarray*}
Hence,
\begin{eqnarray*}
\|u\|_{C^{2,\a}_{\gamma}(M\backslash K_{2})}
&\leq & C\|\Delta_g u\|_{C^{0,\a}_{\gamma-2}(M\backslash K_{1})}+C\left\|\frac{\d u}{\d\eta_g}\right\|_{C^{1,\a}_{\gamma-1}(\Sigma\backslash \d 'K_{1})}
\\
& &\quad +C(\e)\|u\|_{C^{0}_{\gamma}(M\backslash K_{1})}+\e\|u\|_{C^{2,\a}_{\gamma}(M\backslash K_{1})},
\end{eqnarray*}
which implies
\begin{eqnarray}\label{lemma2:5}
\|u\|_{C^{2,\a}_{\gamma}(M\backslash K_{2})}
&\leq & C\|\Delta_g u\|_{C^{0,\a}_{\gamma-2}(M)}+C\left\|\frac{\d u}{\d\eta_g}\right\|_{C^{1,\a}_{\gamma-1}(\Sigma)}
\\
& &\quad +C(\e)\|u\|_{C^{0}_{\gamma}(M)}+\e\|u\|_{C^{2,\a}_{\gamma}(M)}\,.\notag
\end{eqnarray}

Let us now consider a smooth cutoff function $0\leq\theta\leq1$ satisfying $\theta\equiv 1$ in $K_3$ and $\theta\equiv 0$ in $M\backslash K_4$. By elliptic regularity, 
\begin{eqnarray*}
\|u\|_{C^{2,\a}(K_3)}
&\leq  & \|\theta u\|_{C^{2,\a}(K_4)}
\leq C\|\Delta_g(\theta u)\|_{C^{0,\a}(K_4)}+ C\left\|\frac{\d}{\d\eta_g}(\theta u)\right\|_{C^{1,\a}(\d 'K_4)}
\\
&\leq &  C\|\Delta_g  u\|_{C^{0,\a}(K_4)}+ C\left\|\frac{\d u}{\d\eta_g}\right\|_{C^{1,\a}(\d 'K_4)}+C\|u\|_{C^{1,\a}(K_4)}
\\
&\leq & C\|\Delta_g  u\|_{C^{0,\a}(K_4)}+ C\left\|\frac{\d u}{\d\eta_g}\right\|_{C^{1,\a}(\d 'K_4)}+C(\e)\|u\|_{C^{0}(K_4)}\\
& & \quad +\e\|u\|_{C^{2,\a}(K_4)},
\end{eqnarray*}
so that 
\begin{eqnarray}\label{lemma2:6}
\|u\|_{C^{2,\a}(K_3)}
&\leq &  C\|\Delta_g  u\|_{C^{0,\a}_{\gamma-2}(M)}+ C\left\|\frac{\d u}{\d\eta_g}\right\|_{C^{1,\a}_{\gamma-1}(\Sigma)}\\
&& \quad +C(\e)\|u\|_{C^{0}_{\gamma}(M)}
 +\e\|u\|_{C^{2,\a}_{\gamma}(M).}\notag
\end{eqnarray}
The estimate in item (a) follows immediately from
(\ref{lemma2:5}) and (\ref{lemma2:6}).

In order to prove item (b), assume that $g=\delta$ in $M\backslash K_{R}$, $R\geq 1$, and consider a smooth cutoff function $0\leq\theta\leq 1$ satisfying $\theta\equiv 1$ in $K_R$ and $\theta\equiv 0$ in $M\backslash K_{2R}$. Since $(1-\theta)u$ has  support in $M_{\infty}$, the restriction of $(1-\theta)u$ to $M_{\infty}$ can be seen as a function $v\in C^{2}_{\gamma}(\Rn)$. Hence, according to Lemma \ref{lemma1}, there exists $C=C(n)>0$ such that
\[
\|v\|_{C^0_{\gamma}(\Rn)}\leq C\|\Delta v\|_{C^0_{\gamma-2}(\Rn)}+C\left\|\frac{\d v}{\d x_n}\right\|_{C^0_{\gamma-1}(\d\Rn)}.
\]
Since
$
\|u\|_{C^0_{\gamma}(M\backslash K_{2R})}\leq \|v\|_{C^0_{\gamma}(\Rn)}
$
we obtain
\begin{eqnarray*}
\|u\|_{C^0_{\gamma}(M\backslash K_{2R})}
& \leq &  
C\|\Delta_g u\|_{C^0_{\gamma-2}(M\backslash K_{R})}+C\left\|\frac{\d u}{\d\eta_g}\right\|_{C^0_{\gamma-1}(\Sigma\backslash \d 'K_{R})}\\
& & \quad 
+C\|u\|_{C^1(K_{2R}\backslash K_{R})},
\end{eqnarray*}
which clearly implies
the estimate in item (b).
\ep

\begin{lemma}\label{lemma3}
Let $(M,g)$ be as in Lemma \ref{lemma2} and consider the operators $L=\Delta_g+h$ and $B=\frac{\d}{\d\eta_g}+\bar{h}$ where $h\in C^{0,\a}_{-2-\e}(M)$ and $\bar{h}\in C^{1,\a}_{-1-\e}(\Sigma)$. If $2-n<\gamma<0$, we define by $T(u)=(Lu,Bu)$ the operator
$$
T:C^{2,\a}_{\gamma}(M)\to C^{0,\a}_{\gamma-2}(M)\times C^{1,\a}_{\gamma-1}(\Sigma)\,.
$$
If $T$ is injective then there holds 
\begin{equation}\label{lemma3:0}
\|u\|_{C^{2,\a}_{\gamma}(M)}
\leq C\|Lu\|_{C^{0,\a}_{\gamma-2}(M)} +C\|Bu\|_{C^{1,\a}_{\gamma-1}(\Sigma)},
\end{equation}
for all $u\in C^{2,\a}_{\gamma}(M)$ and some $C=C(M,g, \gamma, \|h\|_{C^{0,\a}_{-2-\e}(M)},\|\bar{h}\|_{ C^{1,\a}_{-1-\e}(\Sigma)})>0$.
\end{lemma}
\bp
We retain the notation in the proof of Lemma \ref{lemma2}.
We consider a smooth cutoff function $0\leq\theta\leq 1$ satisfying $\theta\equiv 1$ in $M\backslash K_{2}$  and $\theta\equiv 0$ in $K_1$. Since the support of $\theta$ is contained in $M\backslash K_1$, it makes sense to define $\theta_R$, $R\geq 1$,  by $\theta_R(x)=\theta(R^{-1}x)$, which is supported in $M\backslash K_{R}\subset M_{\infty}$.
Also, 
let us define the metric $g_R$ on $M$ by $g_R=\theta_R\delta +(1-\theta_R)g$, so that $g_R=\delta$ in $M\backslash K_{2R}$. For later use we observe that 
\begin{equation}\label{lemma3:1}
g-g_R=\theta_R(g-\delta).
\end{equation}
By the last assertion in  Lemma \ref{lemma2} there exists $C>0$ such that
\begin{align}\label{lemma3:2}
\|u\|_{C^{2,\a}_{\gamma}(M)}
\leq
C\|\Delta_{g_R}u\|_{C^{0,\a}_{\gamma-2}(M)}+
C\left\|\frac{\partial u}{\d\eta_{g_R}}\right\|_{C^{1,\a}_{\gamma-1}(M)}
+C\|u\|_{C^{1}(K_{R'})}
\end{align}
for some large $R'\geq 2R$.

Let us first estimate $\|(L-\Delta_{g_R})u\|_{C^{0,\a}_{\gamma-2}(M)}$.
Using the standard coordinate expression for the Laplacian we can verify that
\begin{align}
\|(\Delta_{g_R}-\Delta_g)u\|_{C^{0,\a}_{\gamma-2}(M)}
&\leq 
C\|g_R-g\|_{C^{0,\a}_0(M)}\|\nabla_g^2 u\|_{C^{0,\a}_{\gamma-2}(M)}\notag
\\
&\hspace{1cm}+C\|g_R-g\|_{C^{1,\a}_0(M)}\|\nabla_g u\|_{C^{0,\a}_{\gamma-1}(M)}\notag
\\
&\leq C\|g_R-g\|_{C^{1,\a}_0(M)}\|u\|_{C^{2,\a}_{\gamma}(M)}\,.\notag
\end{align} 
Also, using (\ref{lemma3:1}) we have
\begin{eqnarray*}
\|g_R-g\|_{C^{1,\a}_0(M)}
&\leq & 
C\|\theta_R(g-\delta)\|_{C^{1,\a}_0(M)}\notag
\\
&\leq &
C\|g-\delta\|_{C^{1,\a}_0(M\backslash K_R)}\\
& \leq  & CR^{-\tau}\|g-\delta\|_{C^{1,\a}_{-\tau}(M\backslash K_R)}\,,\notag
\end{eqnarray*} 
where the constant $C$ is independent of $R$, which implies
\[
\|(\Delta_{g_R}-\Delta_g)u\|_{C^{0,\a}_{\gamma-2}(M)}
\leq CR^{-\tau}\|u\|_{C^{2,\a}_{\gamma}(M)}.
\]
On the other hand, writing $u=(1-\theta_R) u+\theta_R u$ we have
\begin{eqnarray*}
\|hu\|_{C^{0,\a}_{\gamma-2}(M)}
&\leq & 
\|h(1-\theta_R)u\|_{C^{0,\a}_{\gamma-2}(M)}+\|h\theta_R u\|_{C^{0,\a}_{\gamma-2}(M\backslash K_R)}
\\
&\leq & 
\|h\|_{C^{0,\a}_{-2}(M)}\|(1-\theta_R)u\|_{C^{0,\a}_{\gamma}(M)}+\|h\|_{C^{0,\a}_{-2}(M\backslash K_R)}\|\theta_R u\|_{C^{0,\a}_{\gamma}(M\backslash K_R)}\notag
\\
&\leq &
C\|h\|_{C^{0,\a}_{-2-\e}(M)}\|u\|_{C^{0,\a}(K_{2R})}+R^{-\e}\|h\|_{C^{0,\a}_{-2-\e}(M\backslash K_R)}\|u\|_{C^{0,\a}_{\gamma}(M)}\notag
\\
&\leq & 
C(\|u\|_{C^{0,\a}(K_{2R})}+R^{-\e}\|u\|_{C^{0,\a}_{\gamma}(M)})\,,
\end{eqnarray*}
where the last constant $C$ depends on $\|h\|_{C^{0,\a}_{-2-\e}(M)}$. Thus, 
\begin{align}\label{lemma3:5}
\|(L-\Delta_{g_R})u\|_{C^{0,\a}_{\gamma-2}(M)}
&\leq
C\|u\|_{C^{0,\a}(K_{2R})}+C(R^{-\tau}+R^{-\e})\|u\|_{C^{2,\a}_{\gamma}(M)}\,.
\end{align}

Similarly, if we make use of (\ref{unitnormalexp})
we obtain
\begin{eqnarray}\label{lemma3:8}
\big\|\big(B-\frac{\d}{\d\eta_{g_R}}\big)u\big\|_{C^{1,\a}_{\gamma-1}(\Sigma)}
& \leq & 
C\|u\|_{C^{1,\a}(\d 'K_{2R})} +\\
& & \quad +C(R^{-\tau}+R^{-\e})\|u\|_{C^{2,\a}_{\gamma}(\Sigma)},\notag
\end{eqnarray}
where $C$ depends on $\|\bar{h}\|_{C^{1,\a}_{-1-\e}(\Sigma)}$, so that (\ref{lemma3:2}), (\ref{lemma3:5}) and (\ref{lemma3:8}) lead to 
\begin{eqnarray*}
\|u\|_{C^{2,\a}_{\gamma}(M)}
&\leq &
C\|Lu\|_{C^{0,\a}_{\gamma-2}(M)}+
C\|Lu-\Delta_{g_R}u\|_{C^{0,\a}_{\gamma-2}(M)}\\
& &  \quad +C\|Bu\|_{C^{1,\a}_{\gamma-1}(\Sigma)}+C\big\|Bu-\frac{\d}{\d\eta_{g_R}}u\big\|_{C^{1,\a}_{\gamma-1}(\Sigma)}
\\
& & \quad \quad +C\|u\|_{C^{1}(K_{R'})}
\\
&\leq  &
C\|Lu\|_{C^{0,\a}_{\gamma-2}(M)}+C(R^{-\tau}
+R^{-\e})\|u\|_{C^{2,\a}_{\gamma}(M)}
\\
& &\quad +C\|Bu\|_{C^{1,\a}_{\gamma-1}(\Sigma)}+C\|u\|_{C^{1,\a}(K_{R'})}.
\end{eqnarray*}
Hence, if we choose  $R$  large we finally obtain the key estimate
\begin{align}\label{lemma3:9}
\|u\|_{C^{2,\a}_{\gamma}(M)}
&\leq
C\|Lu\|_{C^{0,\a}_{\gamma-2}(M)}+C\|Bu\|_{C^{1,\a}_{\gamma-1}(\Sigma)}+C\|u\|_{C^{1,\a}(K_{R'})}\,.
\end{align}

The rest of the proof of Lemma \ref{lemma3} will follow by a contradiction argument using the injectivity assumption and (\ref{lemma3:9}).
Indeed, assuming that (\ref{lemma3:0}) does not hold we can choose $\{u_j\}_{j=1}^{\infty}\subset C^{2,\a}_{\gamma}(M)$ satisfying
$$
1=\|u_j\|_{ C^{2,\a}_{\gamma}(M)}\geq j\|Lu_j\|_{ C^{0,\a}_{\gamma-2}(M)}+ j\|Bu_j\|_{ C^{1,\a}_{\gamma-1}(\Sigma)}\,.
$$
In particular, as $j\to\infty$,
\begin{equation}\label{lemma3:10}
\begin{cases}
Lu_j\to 0\:\:\text{in}\:C^{0,\a}_{\gamma-2}(M)\,,
\\
Bu_j\to 0\:\:\text{in}\:C^{1,\a}_{\gamma-1}(\Sigma)\,.
\end{cases}
\end{equation}
Since $\|u_j\|_{ C^{2,\a}_{\gamma}(M)}=1$ we can assume that $\{u_j\}$ converges in $C^{1,\a}(K_{R'})$. Then, using  (\ref{lemma3:9}) with $u=u_j-u_k$, we see that $\{u_j\}$ is a Cauchy sequence in $ C^{2,\a}_{\gamma}(M)$. Hence, this sequence converges in $ C^{2,\a}_{\gamma}(M)$ to some $u\in  C^{2,\a}_{\gamma}(M)$ with $\|u\|_{ C^{2,\a}_{\gamma}(M)}=1$. The fact that $T=(L,B)$ is a continuous operator together with (\ref{lemma3:10}) implies that $Lu=0$ and $Bu=0$. Thus, $u\equiv 0$ by the injectivity hypothesis. This contradicts the fact that $\|u\|_{ C^{2,\a}_{\gamma}(M)}=1$ and concludes the proof of Lemma \ref{lemma3}.
\ep

\begin{lemma}\label{lemma4} Let $(M,g)$ be an asymptotically flat manifold as in Lemma \ref{lemma3}. If $2-n<\gamma<0$ consider the operator
$$
T:C^{2,\a}_{-\gamma}(M)\to C^{0,\a}_{\gamma-2}(M)\times C^{1,\a}_{\gamma-1}(\Sigma)\,.
$$
given by $T(u)=(\Delta_g u, \frac{\d u}{\d\eta_g})$.
If $g=\delta$ outside a compact subset of $M$ then $T$ is an isomorphism.
\end{lemma}

\bp
We use the notation in the proof of Lemma \ref{lemma3}.
We choose $R$ large so that $g=\delta$ in $M\backslash K_R$  and consider the set $M\backslash K_R\subset M_{\infty}$, which we still denote by $M_{\infty}$. Then  the diffeomorphism 
$$
\phi:B^+_{R^{-1}}(0)\backslash\{0\}\to M_{\infty}, \quad \phi(x)=\frac{x}{|x|^2}
$$
extends to a coordinate system $\phi:B^+_{R^{-1}}(0)\to M_{\infty}\cup\{\infty\}$.  Moreover, if $\{\partial_{i}\}_{i=1}^{n}$ is the canonical frame on $B^+_{R^{-1}}(0)$, then $\phi$ induces the coordinate frame $\{\phi_*\partial_i\}_{i=1}^{n}$ on $M_{\infty}\cup\{\infty\}$.

Let us define a metric $\widetilde{g}=\zeta^4 g$ on $M$, where $\zeta$ is a positive smooth function on $M$ satisfying $\zeta(\phi(x))=|x|$ for all $x\in B^+_{R^{-1}}(0)\backslash \{0\}$. Observe that $\widetilde{g}$ can be extended to a smooth metric on $M\cup\{\infty\}= K_R\cup M_{\infty}\cup \{\infty\}$, still denoted by $\widetilde{g}$. In fact, for any $x\in B^+_{R^{-1}}(0)\backslash \{0\}$, we have $\widetilde{g}_{\phi(x)}\big(\phi_*\partial_i,\phi_*\partial_j\big)=\delta_{ij}$.

In this setting,
let us consider the problem of finding $u\in C^{2,\a}_{\gamma}(M)$ satisfying 
\begin{align}\label{lemma4:1}
\begin{cases}
-\Delta_g u=f\,,\:\text{in}\:M\,,
\\
\frac{\d}{\d\eta_g} u=\bar{f}\,,\:\text{on}\:\Sigma\,,
\end{cases}
\end{align}
for given $f\in C^{0,\a}_{\gamma-2}(M)$ and $\bar{f}\in C^{1,\a}_{\gamma-1}(\Sigma)$. 
Let us first assume that $f$ and $\bar{f}$ are compactly supported.

By eventually multiplying 
$f$ and $\bar{f}$ by real constants and using the notation of Section \ref{firstproof}, this is equivalent to 
\begin{align}\notag
\begin{cases}
L_g u-R_g u=f\,,\:&\text{in}\:M\,,
\\
B_g u-H_g u=\bar{f}\,,\:&\text{on}\:\Sigma
\end{cases}
\end{align}
and using that $L_{\widetilde{g}}(\zeta^{2-n}u)=\zeta^{-n-2}L_g u$ and $B_{\widetilde{g}}(\zeta^{2-n}u)=\zeta^{-n}B_g u$, this becomes
\begin{align}\label{lemma4:3'}
\begin{cases}
L_{\widetilde{g}} v-\zeta^{-4}R_g v=\zeta^{-n-2}f\,,\:&\text{in}\:M\,,
\\
B_{\widetilde{g}} v-\zeta^{-2}H_g v=\zeta^{-n}\bar{f}\,,\:&\text{on}\:\Sigma\,,
\end{cases}
\end{align}
where $v=\zeta^{2-n}u$. Now we shall find $v\in C^{\infty}(M\cup\{\infty\})$ solving 
\begin{align}\label{lemma4:3}
\begin{cases}
L_{\widetilde{g}} v-\zeta^{-4}R_g v=\zeta^{-n-2}f\,,\:&\text{in}\:M\cup\{\infty\}\,,
\\
B_{\widetilde{g}} v-\zeta^{-2}H_g v=\zeta^{-n}\bar{f}\,,\:&\text{on}\:\Sigma\cup\{\infty\}\,,
\end{cases}
\end{align}
so that it solves (\ref{lemma4:3'}) in particular. To that end it suffices to prove uniqueness for this last problem. Thus, suppose that $v$ satisfies 
\begin{align}\notag
\begin{cases}
L_{\widetilde{g}} v-\zeta^{-4}R_g v=0\,,\:\text{in}\:M\cup\{\infty\}\,,
\\
B_{\widetilde{g}} v-\zeta^{-2}H_g v=0\,,\:\text{on}\:\Sigma\cup\{\infty\}\,.
\end{cases}
\end{align}
Then elliptic regularity implies that $v\in C^{\infty}(M\cup\{\infty\})$ and we shall see that actually $v\equiv 0$. Since multiplication is continuous in weighted H\"older spaces (see \cite[Lemma 1]{CSCB}), it follows  that $\zeta^{n-2}v\in C^{k}_{2-n}(M)$  for any $k\geq 0$. 
But  $u=\zeta^{n-2}v$  satisfies
\begin{align}\notag
\begin{cases}
\Delta_g u=0\,,\:\text{in}\:M\,,
\\
\frac{\partial}{\d\eta_g} u=0\,,\:\text{on}\:\Sigma\,,
\end{cases}
\end{align}
which implies $u\equiv 0$, as we can check by a simple integration by parts. Hence, $v\equiv 0$ and we have uniqueness for the problem (\ref{lemma4:3}).
(In particular, $T$ is injective.) Thus, we can always find a solution $v\in C^{\infty}(M\cup\{\infty\})$  to  (\ref{lemma4:3}) and hence a solution $u\in C^{2,\a}_{\gamma}(M)$ to (\ref{lemma4:1}) in case both $f$ and $\bar f$ are compactly supported.

Let us  now consider the general case where 
$f\in C^{0,\a}_{\gamma-2}(M)$ and $\bar{f}\in C^{1,\a}_{\gamma-1}(\Sigma)$.  
We want to find $u\in C^{2,\a}_{\gamma}(M)$ such that $(L u,B u)=(f,\bar{f})$. If $0<\a_1<\a$ and $\gamma<\gamma_1<0$ we can find sequences $\{f_j\}_{j=1}^{\infty}\subset C^{\infty}_{c}(M)$ and $\{\bar{f}_j\}_{j=1}^{\infty}\subset C^{\infty}_{c}(\Sigma)$ such that, as $j\to +\infty$, 
$$
\|f_j-f\|_{C^{0,\a_1}_{\gamma_1-2}(M)}\to 0\,,\:\:\:\|\bar{f_j}-\bar{f}\|_{C^{1,\a_1}_{\gamma_1-1}(\Sigma)}\to 0\,,
$$ 
and
$$
\|f_j\|_{C^{0,\a}_{\gamma-2}(M)}\leq C\|f\|_{C^{0,\a}_{\gamma-2}(M)}\,,\:\:\:\|\bar{f_j}\|_{C^{1,\a}_{\gamma-1}(\Sigma)}\leq C|\bar{f}\|_{C^{1,\a}_{\gamma-1}(\Sigma)}\,.
$$
By the special case already proved, 
we can find $u_j\in C^{2,\a}_{-\gamma}(M)$  so that $(L u_j, B u_j)=(f_j,\bar{f}_j)$. It follows from Lemma \ref{lemma3} that, as $j,k\to +\infty$,
\begin{align}
\|u_j-u_k\|_{C^{2,\a_1}_{\gamma_1}(M)}
\leq C\|f_j-f_k\|_{C^{0,\a_1}_{\gamma_1-2}(M)}+C\|\bar{f}_j-\bar{f}_k\|_{C^{1,\a_1}_{\gamma_1-1}(\Sigma)}\to 0\,,\notag
\end{align}
and
\begin{align}
\|u_j\|_{C^{2,\a}_{\gamma}(M)}
\leq C\|f_j\|_{C^{0,\a}_{\gamma-2}(M)}+C\|\bar{f}_j\|_{C^{1,\a}_{\gamma-1}(\Sigma)}
\leq C\,.\notag
\end{align}
Hence, we can assume that $u_j\to u$ in $C^{2,\a_1}_{\gamma_1}(M)$ for some $u\in C^{2,\a}_{\gamma}(M)$.
As a consequence,  $L u_j\to L u$ in $C^{0,\a_1}_{\gamma_1-2}(M)$ and $B u_j\to B u$ in $C^{1,\a_1}_{\gamma_1-1}(\Sigma)$, as $j\to\infty$, and the result follows from the fact that $L u_j=f_j\to f$ in $C^{0,\a_1}_{\gamma_1-2}(M)$ and $B u_j=\bar{f}_j\to \bar{f}$ in $C^{1,\a_1}_{\gamma_1-1}(\Sigma)$.
\ep

\bp[Proof of Proposition \ref{isomorphism}]
First observe that all the operators $T$ as in the proposition are injective, as we can see by applying the maximum principle. 
Let $\widetilde{\mathcal{C}}$ be the set of all these operators and let $\mathcal{C}\subset\widetilde{\mathcal{C}}$ be the subset of isomorphisms. We consider $\widetilde{\mathcal{C}}$ with the operator norm  topology.
It follows from the Implicit Function Theorem that $\mathcal{C}$ is open in $\widetilde{\mathcal{C}}$. We will  prove that it is also closed.

We set $X=C^{0,\a}_{\gamma-2}(M)$,  $Y=C^{1,\a}_{\gamma-1}(\Sigma)$ and consider $X\times Y$ with the norm 
\[
\|(f,\bar{f})\|_{X\times Y}=\|f\|_X+\|\bar{f}\|_Y.
\]
Let $T_j\in \mathcal{C}$  be a sequence converging to some $T\in\widetilde{\mathcal{C}}$ under the operator norm $\|\,\|_{op}$. We shall prove that $T$ is surjective. 

Given $(f,\bar{f})\in X\times Y$ we must find $u\in C^{2,\a}_{\gamma}(M)$ such that $T(u)=(Lu,Bu)=(f,\bar{f})$. 
Let us write $T_j=(L_j,B_j)$. By hypothesis, there exists  $u_j\in C^{2,\a}_{\gamma}(M)$ satisfying $T(u_j)=(L_j u_j,B_j u_j)=(f,\bar{f})$, so that, by Lemma \ref{lemma3}, there exists $C>0$ such that
$$
\|u_j\|_{ C^{2,\a}_{\gamma}(M)}\leq C\|(f,\bar{f})\|_{X\times Y}\,,
$$
for all $j$.
In particular,  $u_j$ is uniformly bounded in $C^{2,\a}_{\gamma}(M)$. 

If we choose $\a_1\in (0,\a)$ and $\gamma_1\in(\gamma, 0)$, it follows from \cite[Lemma 3]{CSCB} that, by eventually passing to a subsequence, we may assume that $u_j\to u$  in $C^{2,\a_1}_{\gamma_1}(M)$, for some $u\in C^{2,\a}_{\gamma}(M)$. 

We just need to prove that $L_j u_j\to Lu$ in $C^0(M)$ and  $B_j u_j\to Bu$ in $C^0(\Sigma)$ to conclude that $(Lu,Bu)=(f,\bar{f})$.
Observe that $\|T-T_j\|_{op}\to 0$ implies that $\|L-L_j\|_{op}\to 0$ and $\|B-B_j\|_{op}\to 0$.
We also have
\begin{align}\label{propo1:1}
\|L_j u_j-Lu\|_{C^0(M)}\leq \|L(u_j-u)\|_{C^0(M)}+\|(L_j-L) u_j\|_{C^0(M)}\,.
\end{align}
The first term on the right-hand side of (\ref{propo1:1}) converges to zero because $u_j\to u$ in $C^{2,\a_1}_{\gamma_1}(M)$. As for the  second one, 
$$
\|(L_j-L) u_j\|_{C^0(M)}
\leq 
\|(L_j-L) u_j\|_{C^{0,\a}_{\gamma-2}(M)}
\leq \|L_j-L\|_{op}\|u_j\|_{C^{2,\a}_{\gamma}(M)}\to 0\,,
$$
since $\|u_j\|_{C^{2,\a}_{\gamma}(M)}$ is uniformly bounded. This proves that $\|L_j u_j-Lu\|_{C^0(M)}\to 0$. The proof that $\|B_j u_j-Bu\|_{C^0(\Sigma)}\to 0$ is similar, which  proves that $\mathcal{C}$ is closed in $\widetilde{\mathcal{C}}$.

Using the notation in the proof of Lemma \ref{lemma3}, we consider the family of metrics $g_R$ for $R\geq 1$, and observe that the operators of the form $(-\Delta_{g_R}, \frac{\d}{\d \eta_{g_R}})$ are isomorphisms, according to Lemma \ref{lemma4}. Thus, we just need to find a continuous family of injective operators connecting  one of those operators to  $T=(-\Delta_{g}u+h, \frac{\d}{\d \eta_{g}}u+\bar{h})$. This is easily accomplished if we set 
$\displaystyle L_t=-\Delta_{g_{(1-t)^{-1}}}+th$, $\displaystyle B_t=\d/\d\eta_{g_{(1-t)^{-1}}}+t\bar{h}$ and define $T_t=(L_t,B_t)$ for $t\in [0,1)$ and $T_1=T$.
\ep



\end{document}